\documentclass[12pt]{amsart}
\pdfoutput=1
\usepackage[letterpaper, margin=1in]{geometry}
\usepackage[utf8]{inputenc}
\usepackage{amsmath,amssymb,latexsym,amsthm}
\usepackage{MnSymbol}
\usepackage{graphicx}
\usepackage{tikz}
\usetikzlibrary{patterns}
\usepackage{tikz-cd}
\usetikzlibrary{positioning}
\usepackage[alphabetic]{amsrefs}
\usepackage{verbatim}
\usepackage{xspace}
\usepackage{enumitem} 
\usepackage{pinlabel}

\usepackage{subcaption}
\usepackage{caption}
\usepackage[export]{adjustbox} 

\ifx\pdfoutput\undefined
\usepackage{hyperref}
\else
\usepackage[pdftex,colorlinks=true,linkcolor=blue,urlcolor=blue,citecolor=violet]{hyperref}
\fi
\usepackage[nameinlink]{cleveref}
\usepackage{tipa}

\newcommand{\Map}{\mathrm{Map}}
\newcommand{\PMap}{\mathrm{PMap}}

\newcommand{\Out}{\mathrm{Out}}

\newcommand{\Homeo}{\mathrm{Homeo}}
\newcommand{\mcg}{\mathrm{Map}^{\pm}}
\newcommand{\R}{\mathbb{R}}

\newcommand{\Z}{\mathbb{Z}}
\newcommand{\Q}{\mathbb{Q}}
\newcommand{\Ends}{\mathrm{Ends}}
\theoremstyle{theorem}
\newtheorem{theorem}{Theorem}
\numberwithin{theorem}{section}
\newtheorem{corollary}[theorem]{Corollary}
\newtheorem{proposition}[theorem]{Proposition}

\newtheorem{lemma}[theorem]{Lemma}

\newtheorem{remark}[theorem]{Remark}
\newtheorem{question}[theorem]{Question}

\newtheorem*{ack}{Acknowledgements}

\newcommand{\comf}[1]{{\color{blue}[FERRAN: #1]}}
\newcommand{\comc}[1]{{\color{cyan}[CHRIS: #1]}}

\title{Fibrations of handlebodies}
\date{\today}

\author[Hern\'andez Hern\'andez]{Jes\'us Hern\'andez Hern\'andez}
\address{Centro de Ciencias Matem\'aticas}
\email{jhdez@matmor.unam.mx}
\urladdr{\url{https://sites.google.com/site/jhdezhdez/}}

\author[Leininger]{Christopher J. Leininger}
\address{Rice University}
\email{cjl12@rice.edu}
\urladdr{\url{https://sites.google.com/view/chris-leiningers-webpage/home}}

\author[Valdez]{Ferr\'an Valdez}
\address{Centro de Ciencias Matem\'aticas}
\email{ferran@matmor.unam.mx}
\urladdr{\url{https://www.matmor.unam.mx/~ferran/}}

\begin{document}

\begin{abstract}
We show that the open genus 2 handlebody admits uncountably-many fibrations over the circle with fiber homeomorphic to the Cantor tree surface with non-conjugate monodromies in the mapping class group.  The construction generalizes to produce uncountably many fibrations of other tame $3$--manifolds with infinite type fibers, including the blooming Cantor tree and many other types of surfaces.
\end{abstract}

\maketitle
\section{Introduction}
	\label{sec:introduction}


Let $V_g$ denote the interior of the genus $g$ handlebody. Our main result is the following.

\newcommand{\MainTheorem}{
For every primitive class $[\omega]\in H^1(V_2,\Z)$ there exist an uncountable family of fibrations $\{V_2\xrightarrow{p_\alpha}\mathbb{S}^1\}_{\alpha\in(0,1)\setminus\mathbb{Q}}$ with monodromy $\phi_\alpha$ such that:
\begin{enumerate}

	\item $[\omega]$ is dual to the fibration: $[\omega]=[p_\alpha^* d\theta]$, where $d\theta$ generates $H^1(\mathbb{S}^1,\Z)$
	\item every fiber of $p_\alpha$ is homeomorphic to the Cantor tree surface $\Sigma_0$, and
	\item if $Cl[\phi_\alpha]$ denotes the conjugacy class of $[\phi_\alpha]$ in $\Map(\Sigma_0)$, then $\{Cl[\phi_\alpha]\}_{\alpha\in(0,1)\setminus\mathbb{Q}}$ is uncountable.
\end{enumerate}
}

\begin{theorem}
	\label{thm:main1}
	\MainTheorem
\end{theorem}

By appropriately blowing up the orbit of a point $z\in\mathbb{S}^1$ by an irrational rotation $\rho_\alpha$ of the circle, $\alpha\in(0,1)\setminus\mathbb{Q}$, Denjoy produced, nowadays classical, examples of homeomorphisms $f_\alpha\in\Homeo(\mathbb{S}^1)$ which leave invariant a Cantor set $C_\alpha$ in the circle and act transitively on the set of complementary intervals. This homeomorphism $f_\alpha$ can then be extended to the $2$--sphere viewed as the suspension of $\mathbb S^1$, and then restricted to an orientation-preserving homeomorphism $F_\alpha$ of the Cantor tree surface.  We establish (2) by showing that the mapping torus $M_{F_\alpha}$ of $F_\alpha$ is homeomorphic to $V_2$. Item (3) follows from  the classification of \emph{Denjoy continua} by Fokkink \cite{Fokkink} (c.f.~Barge and Williams \cite{BargeWilliams}); see Theorem~\ref{thm:fokkink}.  Finally, (1) is a consequence of the fact that the diffeomorphism group of the handlebody acts transitively on the primitive integral cohomology classes.

Theorem~\ref{thm:main1} was inspired by the fact that the mapping torus of an end-periodic homeomorphism on an infinite genus surface with finitely many ends is a {\em tame} $3$--manifold; that is, the interior of compact 3-manifold with non-empty boundary; see \cite{Fenley92}.  A natural question is which other infinite type surfaces admit tame mapping tori.  The key feature of the construction used in the proof of Theorem~\ref{thm:main1} can be isolated and extracted to produce many other homeomorphisms whose mapping tori are tame; see Section~\ref{section:examples}.  For example, this leads to the following.

\newcommand{\TheoremTwo}
{The interior of the genus 4 handlebody $V_4$ has uncountably many fibrations over the circle with fiber homeomorphic to the blooming Cantor tree surface $\Sigma_\infty$, and pairwise non-conjugate monodromies.  Moreover, the fibrations can be taken to define the same primitive integral cohomology class.}
\begin{theorem}
  \label{thm:main2}
\TheoremTwo
\end{theorem}

In the end-periodic case, the ends of the mapping torus are in bijective correspondence with the orbits of ends.  This might suggest that one should never expect to have a tame $3$--manifold fiber over the circle where the fiber has uncountable many ends.  However, Theorems~\ref{thm:main1} and \ref{thm:main2} show that this intuition is wrong, since the end space in those examples is a Cantor set, and hence has uncountably many orbits.  The intution is not far off however, and the next theorem provides the correct link.

\begin{theorem} \label{thm:orbits to ends} Suppose $f \colon S\to S$ is any homeomorphism of a connected surface $S$. If there exists $x \in \Ends(S)$ with a dense orbit, $\overline{\langle f \rangle \cdot x} = \Ends(S)$, then $M_f$ has one end.
\end{theorem}

This theorem is actually a consequence to a more general result concerning the induced map by the inclusion $i\colon S \to M_f$ given by $i(x) = (x,0)$ (see Proposition \ref{P:Ends prop}).

As a consequence, the fact that the action of the monodromies for the fibrations in Theorems~\ref{thm:main1} and \ref{thm:main2} have dense orbits in the end space, means the end space of the $3$--manifolds must be a single point.  While the authors had some vague intuition when they began investigating the homeomorphisms of the Cantor tree and blooming Cantor tree surfaces, admittedly, the formulation of Proposition~\ref{P:Ends prop} and 
 only came after the constructions from the proofs of Theorems~\ref{thm:main1} and \ref{thm:main2} were done.

End-periodic homeomorphisms were classified in the style of Nielsen-Thurston in unpublished work by Handel and Miller, and later explained and elaborated upon by Cantwell-Conlon-Fenley \cite{CanConFen}.  Their mapping tori arise naturally in theory of depth $1$ foliations \cite{Fenley92,Fenley97}.  Even though Denjoy homeomorphisms of the sphere minus a Cantor satisfy (vacuously) the definition of a strongy irreducible end-periodic homeomorphism, their dynamics differ greatly and as such should be treated separately. The Denjoy homeomorphisms of the sphere minus a Cantor set are examples of a {\em tame} homeomorphism of an infinite type surface which is not {\em `extra tame'}, in the sense of Bestvina-Fanoni-Tao \cite{BesFanTao}, and so do not fit into any known classification scheme.  Their mapping tori do not seem to be related to foliations, however.
Moreover, in a natural Gromov-hyperbolic graph it is not a loxodromic: $F_\alpha$ fixes the north pole, hence we have that $F_\alpha$ defines an element in the mapping class group of $\Sigma^*:=\R\setminus C$, where $C$ is a Cantor set.  The ray and loop graphs of $\Sigma^*$ are quasi-isometric infinite-diameter, Gromov hyperbolic metric spaces, see~\cite{Ba18}. The action of $F_\alpha$ on the ray graph is elliptic since the orbit of a ray going directly from the north pole to a point in $C_\alpha$ has diameter 1. In Section~\ref{sec:final thoughts}, we explore some of these topics and open questions.

\begin{ack} The first author was partially supported during the creation of this article by the UNAM-PAPIIT research grants IN114323 and IN101422. The second author would like to thank the other two authors and Centro de Ciencias Matem\'aticas in Morelia for their hospitality during a visit in June 2024, where this work began.  The second author was supported by NSF grant DMS-2305286. The third author was partially supported during the creation of this article by the UNAM-PAPIIT research grant IN101422 and IN106925.
	The third author wants to thank (7 year old) Anika Valdez for comming up, after discussions regarding the content of this paper, with the name \emph{fibroscocho} (pronounced \textipa{[fi.\;{\:b}ros.{"k}o.tS{o}]}) for any fibration $V_2\xrightarrow{p_\alpha}\mathbb{S}^1$ given by Theorem~\ref{thm:main1}. This is a portmanteau for the words in spanish \emph{fibra} (fiber), \emph{bosque} (forest) and \emph{ocho} (eight).
\end{ack}


\section{Denjoy homeomorphisms for spheres}
	\label{section:Denjoy}


Here we collect some preliminary facts we will need.

\subsection{Denjoy homeomorphisms.} \label{Sec:Denjoy circle} Let $\alpha \in [0,1]\setminus\Q$ be fixed and denote by $\mathbb{S}^1$ the unit circle on the complex plane. We briefly recall the construction of a {\em Denjoy homeomorphism} $f_{\alpha} \in \Homeo(\mathbb{S}^1)$.  We refer the reader to \cite{HasselblattKatok} Subsection 4.4.3, for example, for a more detailed description of this construction.

 Let $\rho_{\alpha}$ be the (irrational) rotation of $\mathbb{S}^1$ of angle $\alpha$. It is a well-known fact that every $\rho_\alpha$-orbit is dense.  For each $n \in \Z$ set $x_{n} \colon = \rho_{\alpha}^{n}(1)$ and denote by $I_{n}$ an interval of length $\frac{1}{2^{|n|}}$.  Next we ``blow up" each $x_n$ to $I_n$, effectively replacing it by the interval $I_n$ to produce a new 1-manifold $\mathbb{S}^{1}_{\alpha}$.  Since $\displaystyle{\sum_{n \in \mathbb Z}\frac{1}{2^{|n|}} = 3}$, the $1$--manifold $\mathbb{S}^{1}_{\alpha}$ is naturally isometric to a circle whose length has increased by $3$.  See Figure \ref{Figure:Denjoy1}.

Next, let $P_\alpha:\mathbb{S}^{1}_{\alpha}\to\mathbb{S}^1$ be the map sending each $I_n$ back down to $x_n$.
Given $x \in \mathbb{S}^{1}_{\alpha}$, either $x \in I_n$ for some $n \in \Z$ or $x \in P_\alpha^{-1}(\mathbb{S}^1 \setminus \{ x_n \mid n \in \Z\})$.  Then we define $f_\alpha \colon \mathbb S^1_\alpha \to \mathbb S^1_\alpha$ so that $f_{\alpha}$ maps $I_n$ to $I_{n+1}$ by an orientation preserving affine homeomorphism and $f_{\alpha}(x) = P_\alpha^{-1}(\rho_{\alpha}(P_\alpha(x)))$ if $x \in P_\alpha^{-1}(\mathbb{S}^1 \setminus \{x_n \mid n \in \mathbb Z\})$.  The map $P_\alpha$ defines a semi-conjugacy from $f_\alpha$ to $\rho_\alpha$; that is, $\rho_\alpha \circ P_\alpha = P_\alpha \circ f_\alpha$.  Finally, we note that $f_{\alpha}$ has $I_0$ as a wandering interval, and it acts minimally on $\mathbb{S}^{1}_{\alpha} \setminus \left(\bigcup_{n \in \Z} I_n\right)$.  Consequently, $\mathbb{S}^{1}_{\alpha} \setminus \left(\bigcup_{n \in \Z} I_n^o\right)$, where $I_n^o$ denotes the interior of $I_n$, is homeomorphic to a Cantor set that we denote by $C_{\alpha}$.


\begin{figure}[h]
    \includegraphics[scale=.35]{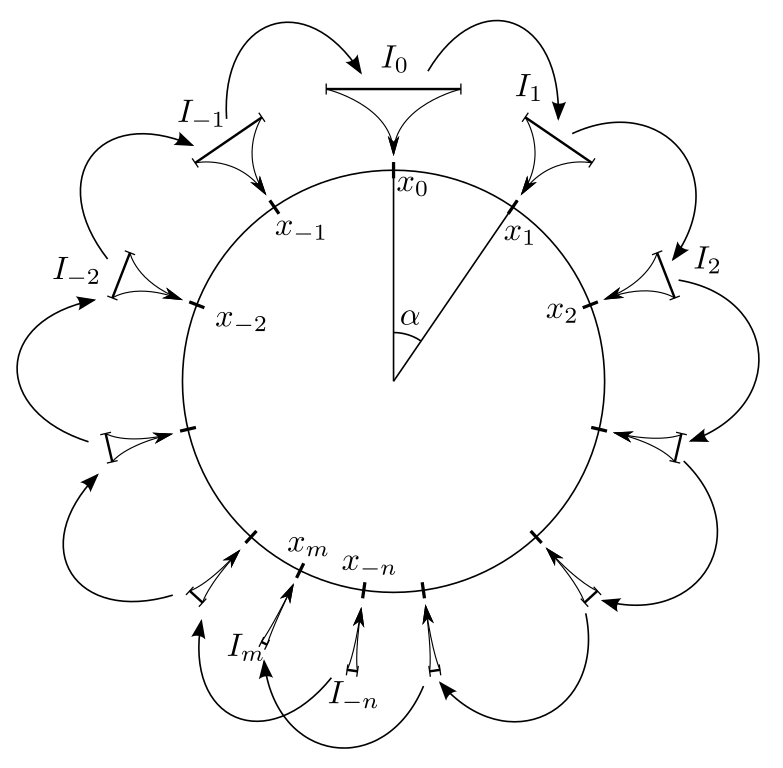}
    \caption{Denjoy construction for an irrational rotation leaving an invariant Cantor set. Figure due to Ilya Voyager, taken from \href{https://es.m.wikipedia.org/wiki/Archivo:Denjoy-example.svg}{Wikipedia Commons} .}\label{Figure:Denjoy1}
\end{figure}

\subsection{Extending Denjoy homeomorphisms to $\mathbb{S}^2$.} \label{Sec:Denjoy sphere} Consider the quotient of the cylinder $\mathbb{S}_{\alpha}^1 \times [-1,1]$ by the identifications $(x,-1) \sim (y,-1)$ and $(x,1) \sim (y,1)$ for all $x,y \in \mathbb S_{\alpha}^1$.  Define  $\mathbb S_{\alpha}^2:=\mathbb S_{\alpha}^1 \times [-1,1]/\sim$, and denote the associated quotient map by
\[ \pi_\alpha \colon \mathbb{S}_{\alpha}^1 \times [-1,1] \to \mathbb S_{\alpha}^2.\]
Note that $ \mathbb S_{\alpha}^2 $ is homeomorphic to $\mathbb S^2$. We define $\widetilde{F}_\alpha \colon \mathbb S^2_\alpha \to \mathbb S^2_\alpha$ to be the homeomorphism which is the descent of $f_\alpha \times \mbox{id}_{[-1,1]}$ to $\mathbb S^2_\alpha$ by $\pi_\alpha$.  That is, $\widetilde F_\alpha$ is the unique map defined by $\widetilde{F}_\alpha(\pi_\alpha(x,t)) = \pi_\alpha(f_\alpha(x),t)$.



The {\em north} and {\em south poles} of $\mathbb S_{\alpha}^2$ are the points $N = \pi_\alpha(\mathbb S_{\alpha}^1 \times \{1\})$ and $S = \pi_\alpha(\mathbb S_{\alpha}^1 \times \{-1\})$. These poles form the set of fixed points of  $\widetilde{F}_{\alpha}$. Note that $\widetilde{F}_{\alpha}$ leaves invariant every parallel of the sphere, $\pi_\alpha(\mathbb S^1_\alpha \times \{\ast\}) \subset \mathbb{S}_{\alpha}^2$. The equator is one such parallel and we write
\[ \mathbb S^1_\alpha \subset \mathbb S^2_\alpha,\]
with the identification coming from the obvious homeomorphism $\mathbb S^1_\alpha \cong \mathbb S^1_\alpha \times \{0\}$ composed with $\pi_\alpha$.   With this identification, $\widetilde{F}_{\alpha}$ restricted to the equator is precisely $f_\alpha$, and so the equatorial Cantor set, $C_\alpha \subset \mathbb S^1_\alpha \subset \mathbb S^2_\alpha$, is also $\widetilde F_\alpha$--invariant. Thus, $\widetilde{F}_{\alpha}$ restricts to a self-homeomorphism $F_{\alpha}$ of the surface $\Sigma_\alpha:=\mathbb{S}_{\alpha}^2 \setminus C_{\alpha}$,
\[
  F_\alpha \colon \Sigma_\alpha \to \Sigma_\alpha.
\]
By the Kerékjártó-Richards classification of surfaces (see \cite{Kerekjarto} and \cite{Richards63}), $\Sigma_\alpha$ and $\Sigma_\beta$ are homeomorphic surfaces and both are homeomorphic to the Cantor tree surface $\Sigma_0$: the surface obtained from gluing countably many pairs of pants in the pattern given by the 3-regular tree.  Choose a homeomorphism $\Sigma_\alpha \to \Sigma_0$ and conjugate $F_\alpha$ to a homeomorphism $\phi_\alpha \colon \Sigma_0 \to \Sigma_0$.  Any other choice of homeomorphism gives rise to conjugate homeomorphisms of $\Sigma_0$.  In particular, while $\phi_\alpha$ depends on the choice of homeomorphism $\Sigma_\alpha \to \Sigma_0$, the conjugacy class in the mapping class group $Cl([\phi_\alpha]) \subset \Map(\Sigma_0)$ is independent of this choice.

\subsection{Distinguishing conjugacy classes of Denjoy homeomorphisms of spheres}

We will need a classical result about \emph{Denjoy continua}.  A Denjoy continuum, $\mathbb{D}_\alpha$, is the suspension of the Denjoy homeomorphism $f_\alpha$ restricted to the Cantor set $C_\alpha$. More precisely:
$$
\mathbb{D}_\alpha:=C_\alpha\times[0,1]/(t,1)\sim(f_\alpha(t),0).
$$
The group ${\rm GL}(2,\Z)$ acts on $\mathbb{R}\setminus\Q$ by $\big(\begin{smallmatrix}
  a & b\\
  c & d
\end{smallmatrix}\big)\cdot\alpha=\frac{a\alpha+b}{c\alpha+d}$. We say that two irrational numbers $\alpha$ and $\alpha'$ are equivalent if $\alpha'\in{\rm GL}(2,\Z)\cdot\alpha$.

\begin{remark} In \cite{BargeWilliams}, the authors write ${\rm SL}(2,\Z)$ to denote the group of integer, $2 \times 2$ matrices with determinant $\pm 1$; we have chosen to use the more common notation ${\rm GL}(2,\Z)$ for this group.
\end{remark}

\begin{theorem}[\cite{Fokkink}, \cite{BargeWilliams}]
	\label{thm:fokkink}
Given $\alpha,\alpha' \in (0,1) \setminus \mathbb Q$, their Denjoy continua are homeomorphic, $\mathbb{D}_\alpha \cong \mathbb{D}_{\alpha'}$, if and only if $\alpha$ and $\alpha'$ are equivalent.
\end{theorem}

As a consequence, we have the following.
\begin{corollary} \label{cor:conjugate equivalent}
The set $\{Cl([\phi_\alpha])\}_{\alpha \in (0,1)\setminus \mathbb Q} \subset \Map(\Sigma_0)$ consists of uncountably many conjugacy classes.
\end{corollary}
\begin{proof}
We claim that it is sufficient to show that if $[\phi_\alpha]$ and $[\phi_{\alpha'}]$ are conjugate then ${\mathbb D}_\alpha $ and ${\mathbb D}_{\alpha'}$ are homeomorphic. Indeed, using Theorem~\ref{thm:fokkink} we get that $\alpha$ and $\alpha'$ are equivalent, hence they belong to the same (at most countable) ${\rm GL}(2,\Z)$-orbit. Hence the set of parameters $\alpha'\in (0,1) \setminus \mathbb Q$ such that $[\phi_{\alpha'}]$ belongs to the conjugacy class of $[\phi_\alpha]$ is at most countable. Since $(0,1) \setminus \mathbb Q$ is uncountable, the set $\{Cl([\phi_\alpha])\}_{\alpha \in (0,1)\setminus \mathbb Q}$ consists of uncountably many conjugacy classes.

Now suppose that $[\phi_\alpha] = [f][\phi_{\alpha'}][f]^{-1} = [f\phi_{\alpha'}f^{-1}]$ for some $[f]\in\Map(\Sigma_0)$. Any homeomorphism $h \colon \Sigma \to \Sigma'$ induces a homeomorphism $h_* \colon \Ends(\Sigma) \to \Ends(\Sigma')$.  Moreover, this induced homeomorphism depends only on the isotopy class of $h$.  In particular, $[h] \mapsto h_*$ defines a homomorphism $\Map(\Sigma_0) \to \Homeo(\Ends(\Sigma_0))$.  It follows that ${\phi_\alpha}_* = (f\phi_{\alpha'}f^{-1})_* = f_* {\phi_{\alpha'}}_* f_*^{-1}$ in $\Homeo(\Ends(\Sigma_0))$


  The homeomorphism $f_* \times \mbox{id}_{[0,1]}$ on $\Ends(\Sigma_0) \times [0,1] \to \Ends(\Sigma_0) \times [0,1]$ descends to a homeomorphism from the mapping torus of ${\phi_\alpha}_*$ to the mapping torus of ${\phi_{\alpha'}}_*$.

  On the other hand, the homeomorphism $\Sigma_\alpha \to \Sigma_0$ conjugating $F_\alpha$ to $\phi_\alpha$ extends to a homeomorphism conjugating ${F_\alpha}_* = (f_\alpha)|_{C_\alpha}$ to ${\phi_\alpha}_*$.  As above, the mapping torus ${\mathbb D}_\alpha$ of $(f_\alpha)|_{C_\alpha}$ is therefore homeomorphic to that of ${\phi_\alpha}_*$.  Similarly, ${\mathbb D}_{\alpha'}$ is homeomorphic to the mapping torus of ${\phi_{\alpha'}}_*$.  Combining all these homeomorphisms, we see that ${\mathbb D}_\alpha $ is homeomorphic to ${\mathbb D}_{\alpha'}$.
\end{proof}

This application of Theorem~\ref{thm:fokkink} can be abstracted to other surfaces with no change in the proof.
\begin{corollary} \label{cor:general conjugate equivalent}
If $\Sigma$ is any surface such that $\Ends(\Sigma)$ contains a $\Map(\Sigma)$--invariant Cantor set $C \subset \Ends(\Sigma)$, and if $[\varphi_\alpha],[\varphi_{\alpha'}] \in \Map(\Sigma)$ are conjugate mapping classes such that $(\varphi_\alpha)_*|C$ is conjugate to $(f_\alpha)|_{C_\alpha}$ and $(\varphi_{\alpha'})_*|C$ is conjugate to $(f_{\alpha'})|_{C_\alpha}$, then $\alpha$ and $\alpha'$ are equivalent.
\end{corollary}

\section{Proof of the Main Theorem}
	\label{section:proofs}

In this section we prove our main theorem.

\medskip

\noindent
{\bf Theorem~\ref{thm:main1}}
{\em \MainTheorem
}

\medskip

Let $M_{F_{\alpha}}$ be the mapping torus of $F_{\alpha}$; that is, the quotient space $\Sigma_\alpha \times [0,1]/\sim$ where $(x,1) \sim (F_\alpha(x),0)$, for all $x \in \Sigma_\alpha$.  Theorem~\ref{thm:main1} will follow quickly from the next proposition.


\begin{proposition} \label{Prop:mapping torus handlebody}
For every $\alpha\in[0,1]\setminus\mathbb{Q}$ the mapping torus $M_{F_\alpha}$ is homeomorphic to the interior of the handlebody $V_2$.
\end{proposition}

We postpone the proof of this proposition to the next subsection, but use it to prove the main theorem.

\begin{proof}[Proof of Theorem~\ref{thm:main1} assuming Proposition~\ref{Prop:mapping torus handlebody}] The mapping torus of $F_\alpha$ admits a fibration over the circle obtained from the projection onto the second factor,
\[ M_{F_\alpha} = \Sigma_\alpha \times [0,1]/\sim \, \longrightarrow  \, [0,1]/(0\sim 1) \, \cong  \, \mathbb{S}^1.\]
The dual cohomology class is primitive since the fiber is connected (see e.g.~\cite{ThurstonNorm}).
For any $\alpha$, Proposition~\ref{Prop:mapping torus handlebody} gives us a homeomorphism $V_2 \to M_{F_\alpha}$, and thus $V_2$ admits a fibration $p_\alpha \colon V_2 \to \mathbb{S}^1$.  The natural homomorphism $\Homeo(V_2) \to Out(\pi_1(V_2))$ is surjective (see~\cite{Zieschang61,McMillan63,Griffiths63}), from which it follows that the action of $\Homeo(V_2)$ on primitive integral cohomology classes in $H^1(V_2,\Z) \cong Hom(\pi_1(V_2),\Z)$ is transitive.  In particular, the dual cohomology class of the fibration $p_\alpha: V_2 \to S^1$ can be taken to be any primitive integral cohomology class, proving (1).
A homeomorphism $\Sigma_\alpha \cong \Sigma_0$ to the Cantor tree surface allows us to identify the fiber of $p_\alpha \colon V_2 \to S^1$ with $\Sigma_0$, proving (2).  The conjugating homeomorphism can be chosen to conjugate the monodromy $F_\alpha$ to $\phi_\alpha$.  The collection of these fibrations $\{p_\alpha\}_{\alpha \in (0,1) \setminus \mathbb Q}$ gives uncountably many fibrations.
Corollary~\ref{cor:conjugate equivalent} implies $\{Cl([\phi_\alpha])\}_{\alpha \in (0,1) \setminus \mathbb Q}$ is uncountable, proving (3).
\end{proof}

\begin{remark}
By construction, for every $\alpha\in[0,1]\setminus\mathbb{Q}$ we have an embedding $V_2\hookrightarrow\mathbb{S}^2\times\mathbb{S}^1$. The complement of the image of this embedding is homeomorphic to the Denjoy continum $\mathbb{D}_\alpha$. Moreover, from Theorem~\ref{thm:main1} we obtain \emph{a posteriori} that there exist uncountably many non-homeomorphic Denjoy continua in $\mathbb{S}^2\times\mathbb{S}^1$ whose complement is $V_2$.
\end{remark}

The remainder of the section is devoted to proving Proposition~\ref{Prop:mapping torus handlebody}.

\subsection{Mapping torus of a Denjoy sphere homeomorphism}

Before we prove Proposition~\ref{Prop:mapping torus handlebody}, it is instructive to compute $\pi_1(M_{F_{\alpha}}) \cong F_2$.  Although this fact is a consequence of Proposition~\ref{Prop:mapping torus handlebody}, and in fact the proof of the proposition is logically independent, the calculation guides the ideas in the proof of the proposition.

Recall that $\pi_\alpha \colon \mathbb S^1_\alpha \times [-1,1] \to \mathbb S^2_\alpha$ is the quotient map to the suspension of $\mathbb S^1_\alpha$, and $N,S \in \mathbb S^2_\alpha$ are the north and south poles (i.e.~the suspension points).

Let $\Gamma$ be a connected (geometric) graph with two vertices, $\nu$ and $\sigma$ and countably-many edges, each connecting $\nu$ and $\sigma$, which we index by the integers, $\{e_j\}_{j \in \mathbb Z}$.  See Figure~\ref{Fig:Spine}.  For all $j \in \mathbb Z$, let $y_j$ be the midpoint of the interval $I_j$, and we construct a continuous injection
\[ J_\alpha \colon \Gamma \to \mathbb S_{\alpha}^2,\]
by setting $J_\alpha(\nu) = N$, $J_\alpha(\sigma)=S$ and mapping $e_j$ homeomorphically to $\pi_\alpha(\{y_j\} \times (-1,1))$ for all $j \in \mathbb Z$.

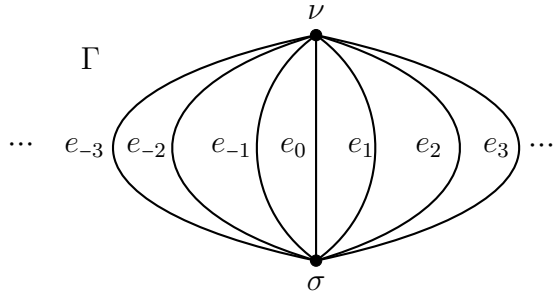
\begin{figure}[h]
\begin{tikzpicture}[scale = 1.5]
\draw[fill=black] (0,1) circle (.05cm);
\draw[fill=black] (0,-1) circle (.05cm);
\draw[thick] (0,1) -- (0,-1);
\draw[thick] (0,1) .. controls (.7,.5) and (.7,-.5) .. (0,-1);
\draw[thick] (0,1) .. controls (-.7,.5) and (-.7,-.5) .. (0,-1);
\draw[thick] (0,1) .. controls (1.7,.5) and (1.7,-.5) .. (0,-1);
\draw[thick] (0,1) .. controls (-1.7,.5) and (-1.7,-.5) .. (0,-1);
\draw[thick] (0,1) .. controls (2.4,.5) and (2.4,-.5) .. (0,-1);
\draw[thick] (0,1) .. controls (-2.4,.5) and (-2.4,-.5) .. (0,-1);
\node at (2,0) {$\cdots$};
\node at (-2.3,0) {$\cdots \quad e_{-3}$};
\node at (-.2,0) {$e_0$};
\node at (.4,0) {$e_1$};
\node at (1,0) {$e_2$};
\node at (1.6,0) {$e_3$};
\node at (-.75,0) {$e_{-1}$};
\node at (-1.5,0) {$e_{-2}$};
\node at (0,1.2) {$\nu$};
\node at (0,-1.2) {$\sigma$};
\node at (-2,.8) {$\Gamma$};
\end{tikzpicture}
\caption{A ``spine" of $\Sigma_\alpha$.}
\label{Fig:Spine}
\end{figure}

\begin{lemma} \label{Lemma:spine} For any $\alpha \in [0,1]\setminus \mathbb Q$, $J_{\alpha *} \colon \pi_1(\Gamma,\nu) \to \pi_1(\Sigma_\alpha,N)$ is an isomorphism.
\end{lemma}

\begin{proof} For any $k > 0$, consider the subgraph $\Gamma_k \subset \Gamma$ spanned by the edges $\{e_j \mid -k \leq j \leq k\}$.  We claim that that the restriction of $J_\alpha$ to $\Gamma_k$ induces an injective homomorphism. Indeed, setting $X_k \subset C_\alpha \subset S^2_\alpha$ to be a set of $2k+1$ points separating the intervals $\{I_j \mid -k \leq j \leq k\}$, there is an inclusion $\Sigma_\alpha \to S^2_\alpha - X_k$ and the induced map $\pi_1(\Sigma_\alpha) \to \pi_1(S^2_\alpha - X_k)$ is surjective.
  The composition of $J_\alpha|_{\Gamma_k}$ with this inclusion is easily seen to be a homotopy equivalence: the image is a spine for $\mathbb S_{\alpha}^2 \setminus X_k$.  It follows that $(J_\alpha|_{\Gamma_k})_* \colon \pi_1(\Gamma_k) \to \pi_1(\Sigma_\alpha)$ is injective.  Since $\pi_1(\Gamma)$ is the direct limit of the subgroups $\pi_1(\Gamma_k)$ as $k \to \infty$, it follows that $J_{\alpha*}$ is injective.

Given any element of $\pi_1(\Sigma_\alpha,N)$, we first represent it by a loop $\gamma \colon [0,1] \to \Sigma_\alpha$ that is smooth.  By a small homotopy, we may assume it intersects $\mathbb S_{\alpha}^1$ transversely, and hence meets only finitely many complementary intervals, and meets each one in finitely many points.  The preimage of these points partitions $[0,1]$ so that the restriction of $\gamma$ to each subinterval is a path connecting points in $(\mathbb S_{\alpha}^1 \setminus C_\alpha)\cup\{N\}$ whose interior is disjoint from the equator $\mathbb S^1_\alpha$.  By further homotopy, we can assume that the endpoints of these paths are in the set $\{y_j\}_{j \in \mathbb Z} \cup \{N,S\}$, and then after another homotopy, that they lie entirely in the image $J(\Gamma)$.  Thus, the element of $\pi_1(\Sigma_\alpha,N)$ is in $J_{\alpha*}(\pi_1(\Gamma,\nu))$, proving that $J_{\alpha*}$ is surjective, and hence an isomorphism.
\end{proof}

\begin{proposition} For any $\alpha \in (0,1) \setminus \mathbb Q$, $\pi_1(M_{F_\alpha}) \cong F_2$.
\end{proposition}
\begin{proof}
Let $F \colon \Gamma \to \Gamma$ be a homeomorphism fixing $\nu$ and $\sigma$ so that $F(e_j) = e_{j+1}$ for all $j \in \mathbb Z$.  Recall from \S\ref{Sec:Denjoy circle} that $f_\alpha$ is an affine map from $I_j$ to $I_{j+1}$, and hence we have $f_\alpha(y_j) = y_{j+1}$, and therefore we may choose $F$ so that $J_\alpha \circ F = F_\alpha \circ J_\alpha$.

Now consider the mapping torus $M_F =  \Gamma \times [0,1] /\sim$, with $(x,1) \sim (F(x),0)$.
The injection $J_\alpha$ extends to an injection of mapping tori, $\bar J_\alpha \colon M_F \to M_{F_\alpha}$.  It follows from Lemma~\ref{Lemma:spine} that ${\bar J}_{\alpha*} \colon \pi_1(M_F) \to \pi_1(M_{F_\alpha})$ is an isomorphism.  Now observe that
\[ \pi_1(M_{F_\alpha}) \cong \pi_1(M_F) \cong \pi_1(\Gamma) \rtimes \mathbb Z \cong F_2,\]
where the final isomorphism comes from the fact that semi-direct product structure is defined by the action of $\mathbb Z$ on the infinitely generated free group $\pi_1(\Gamma) \cong \langle \{x_k\}_{k \in \mathbb Z} \} \mid - \rangle$ so that the generator $1$ of $\mathbb Z$ acts by $x_k \mapsto x_{k+1}$.
\end{proof}

We exploit the ideas in the proof of the previous proposition, and the structure of $M_F$, to better understand $M_{F_\alpha}$.
To this end, we describe a natural cell structure for $M_F$.  Since $\nu$ and $\sigma$ are fixed, $(\nu,0) \sim (\nu,1)$ and $(\sigma,0) \sim (\sigma,1)$ project to two points in $M_F$ which will serve as vertices for a cell structure, and  we denote these points by $\bar \nu$ and $\bar \sigma$, respectively.  There are $1$--cells with characteristic maps obtained by composing $[0,1] \to \nu \times [0,1]$ and $[0,1] \to \sigma \times [0,1]$ with the projection to the quotient space, $\Gamma \times [0,1] \to M_F$, and we denote these $1$--cells by $e_\nu$ and $e_\sigma$.  Each $1$--cell $e_j$ defines a $1$--cell we denote $\bar e_j$ which is the image of $e_j \times \{0\}$.  Finally, $\bar e_,\bar e_{j+1}, \bar e_\nu, \bar e_\sigma$ span a $2$--cube.  See Figure~\ref{Fig:cubulation}.

\begin{figure}[h]
\begin{tikzpicture}
\draw[thick] (-1,.9) circle (.5);
\draw[thick] (-1,-.9) circle (.5);
\draw[fill=black] (-1,.4) circle (.05cm);
\draw[fill=black] (-1,-.4) circle (.05cm);
\node at (-1,.18) {$\bar \nu$};
\node at (-1,-.18) {$\bar \sigma$};
\node at (-1.8,.9) {$\bar e_\nu$};
\node at (-1.8,-.9) {$\bar e_\sigma$};
\draw[thick] (1.8,.5) -- (10.2,.5);
\draw[thick] (1.8,-.5) -- (10.2,-.5);
\draw[thick] (2,.5) -- (2,-.5);
\draw[thick] (3,.5) -- (3,-.5);
\draw[thick] (4,.5) -- (4,-.5);
\draw[thick] (5,.5) -- (5,-.5);
\draw[thick] (6,.5) -- (6,-.5);
\draw[thick] (7,.5) -- (7,-.5);
\draw[thick] (8,.5) -- (8,-.5);
\draw[thick] (9,.5) -- (9,-.5);
\draw[thick] (10,.5) -- (10,-.5);
\filldraw[opacity=.2] (1.85,.5) -- (10.15,.5) -- (10.15,-.5) -- (1.85,-.5) -- (1.85,.-.5);
\node at (10.5,0) {$\cdots$};
\node at (1,0) {$\cdots$};
\node at (2.5,.7) {$\bar e_\nu$};
\node at (2.5,-.7) {$\bar e_\sigma$};
\node at (3.5,.7) {$\bar e_\nu$};
\node at (3.5,-.7) {$\bar e_\sigma$};
\node at (4.5,.7) {$\bar e_\nu$};
\node at (4.5,-.7) {$\bar e_\sigma$};
\node at (5.5,.7) {$\bar e_\nu$};
\node at (5.5,-.7) {$\bar e_\sigma$};
\node at (6.5,.7) {$\bar e_\nu$};
\node at (6.5,-.7) {$\bar e_\sigma$};
\node at (7.5,.7) {$\bar e_\nu$};
\node at (7.5,-.7) {$\bar e_\sigma$};
\node at (8.5,.7) {$\bar e_\nu$};
\node at (8.5,-.7) {$\bar e_\sigma$};
\node at (9.5,.7) {$\bar e_\nu$};
\node at (9.5,-.7) {$\bar e_\sigma$};
\node at (1.65,0) {$\bar e_{-2}$};
\node at (2.72,0) {$\bar e_{-1}$};
\node at (3.8,0) {$\bar e_0$};
\node at (4.8,0) {$\bar e_1$};
\node at (5.8,0) {$\bar e_2$};
\node at (6.8,0) {$\bar e_3$};
\node at (7.8,0) {$\bar e_4$};
\node at (8.8,0) {$\bar e_5$};
\node at (9.8,0) {$\bar e_6$};
\end{tikzpicture}
\caption{The mapping torus $M_F$ and its cell structure.}
\label{Fig:cubulation}
\end{figure}
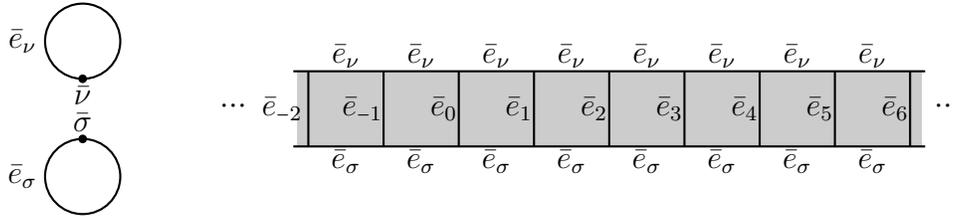

From this description of $M_F$, we can see it as the union of two circles with an infinite strip $[0,1] \times \mathbb R$ attached so that each boundary line is attached by wrapping around the circles infinitely many times in both directions (i.e.~the attaching map is the universal covering).  We get an exhaustion of $M_F$ by compact subsets whose inclusion is a homotopy equivalence, by attaching only compact subsets of the strips.  That is, we exhaust $[0,1] \times \mathbb R$ with rectangles $[0,1] \times [a,b]$, and then $M_F$ by the spaces obtained by attaching only these rectangles to the two circles.

The proof of Proposition~\ref{Prop:mapping torus handlebody} is obtained by ``thickening" the exhaustion of $M_F$ above inside $M_{F_\alpha}$.  The injective map $J_\alpha \colon \Gamma \to \Sigma_\alpha$ and its extension $\bar J_\alpha \colon M_F \to M_{F_\alpha}$ are not proper embeddings, and so this thickening requires some care.  It will be helpful to give a name to the quotient map from the product to the mapping torus,
\[ \pi_M \colon \Sigma_\alpha \times [0,1] \to M_{F_\alpha}.\]
As an abuse of notation (and of topology!), we will identify $M_F$ inside $M_{F_\alpha}$ via ${\bar J}_\alpha$.


We will describe some parts of our construction in terms of attaching handles, so we recall the terminology here to streamline the discussion that follows; see \cite[Sections 1-2]{Scharlemann}.  Given an integer $k \geq 0$, recall that a {\em ($3$--dimensional) $k$--handle} is the product of a $(3-k)$--dimensional disk and a $k$--dimensional disk, $\mathbb D^{3-k} \times \mathbb D^k$.  Given a $3$--manifold with non-empty boundary, $M$, we can attach a $k$--handle to $M$ by specifying an attaching map $\varphi \colon \mathbb D^{3-k} \times \partial \mathbb D^k \to \partial M$, which is a embedding, and then taking the quotient
\[ M \cup_\varphi \mathbb D^{3-k} \times \mathbb D^k = (M \sqcup_\varphi \mathbb D^{3-k} \times \mathbb D^k)/(x \sim \varphi(x)),\]
identifying every $x \in \mathbb D^{3-k} \times \partial \mathbb D^k$ with its $\varphi$--image.  A $1$--handle is $\mathbb D^2 \times [-1,1]$ and these are attached along $\varphi(\mathbb D^2 \times \{\pm 1\})$, a pair of disks in $\partial M$.  A $2$--handle is $[-1,1] \times \mathbb D^2$, and is attached along $\varphi([-1,1] \times \partial \mathbb D^2)$, an annulus in $\partial M$.


\begin{proof}[Proof of Proposition~\ref{Prop:mapping torus handlebody}]
Note that the interior of $V_2$ can be exhausted by $(H_{t})_{t \in (0,1)}$ where each $H_{i}$ is a closed genus $2$ handlebody and the ``gaps", $\overline{H_{t'}\setminus H_t}$, over all $t<t'$, are each homeomorphic to the ``thick'' genus 2 surface $S_2 \times [t,t']$. Our goal is to construct an exhaustion of $M_{F_{\alpha}}$ of the same form.  

Given any $t \in (0,1)$ and $j \in \mathbb Z$, let $tI_j \subset I_j$ denote the subinterval of $I_j$ centered at $y_j$ of length $t|I_j|$.  For all such $t$ and $j$, we use the quotient map $\pi_\alpha \colon \mathbb S^1_\alpha \times [-1,1] \to \mathbb S^2_\alpha$ and construct the following sets (see Figure \ref{Figure:Nbhds of sphere}):
\begin{itemize}
    \item $D_{N,t} = \pi_\alpha(\mathbb{S}_{\alpha}^1 \times [1-t,1])$,
    \item $D_{S,t} = \pi_\alpha(\mathbb{S}_{\alpha}^1 \times [-1,-1+t])$,
    \item $B_{j,t} =\pi_\alpha(tI_j\times[-1,1])$, and
    \item $I_{j,t}^\pm = \pi_\alpha(tI_j \times \{\pm (1-t)\})$.
\end{itemize}

\begin{figure}[ht]
    \centering
    \begin{tikzpicture}[x=0.75pt,y=0.75pt,yscale=-1,xscale=1]
    	\draw    (11.5,204.25) .. controls (90.5,234.75) and (200.5,234.25) .. (280.5,203.75) ;
    	\draw    [dash pattern={on 0.84pt off 2.51pt}](11.5,204.25) .. controls (51.5,174.25) and (240.5,173.75) .. (280.5,203.75) ;
    	\draw    (145.5,33.75) -- (11.5,204.25) ;
    	\draw    (145.5,33.75) -- (280.5,203.75) ;
    	\draw    (11.5,204.25) -- (144,374.25) ;
    	\draw    (279.5,201.25) -- (144,374.25) ;
    	\draw    (121.73,64) .. controls (132.45,76.18) and (159,75.82) .. (169.55,63.82) ;
    	\draw    [dash pattern={on 0.84pt off 2.51pt}](121.73,64) .. controls (138.45,56.91) and (151.55,56.73) .. (169.55,63.82) ;
    	\draw    (120.45,344) .. controls (133.55,355.64) and (156.09,355.27) .. (168.09,343.82) ;
    	\draw    [dash pattern={on 0.84pt off 2.51pt}](120.45,344) .. controls (129.91,333.45) and (160.09,333.82) .. (168.09,343.82) ;
    	\draw [line width=2.25]    (121.75,226.44) .. controls (135.38,227.31) and (165.5,227.69) .. (178.13,225.44) ;
    	\draw    (122.06,223.41) -- (121.82,229.29) ;
    	\draw    (177.71,222.24) -- (178.53,228.94) ;
    	\draw [color={rgb, 255:red, 155; green, 155; blue, 155 }  ,draw opacity=1 ]   (145.5,33.75) -- (151.29,227.43) ;
    	\draw [color={rgb, 255:red, 155; green, 155; blue, 155 }  ,draw opacity=1 ]   (151.29,227.43) -- (144,374.25) ;
    	\draw    (145.5,33.75) -- (140.14,227.14) ;
    	\draw    (140.14,227.14) -- (144,374.25) ;
    	\draw    (145.5,33.75) -- (161,227.43) ;
    	\draw    (161,227.43) -- (144,374.25) ;
    	\draw (128,14) node [anchor=north west][inner sep=0.75pt]   [align=left] {$\displaystyle N$};
    	\draw (128,377.5) node [anchor=north west][inner sep=0.75pt]   [align=left] {$\displaystyle S$};
    	\draw (162,33.4) node [anchor=north west][inner sep=0.75pt]    {$D_{N}{}_{,}{}_{t}$};
    	\draw (168.33,356.4) node [anchor=north west][inner sep=0.75pt]    {$D_{S}{}_{,}{}_{t}$};
    	\draw (124.06,230) node [anchor=north west][inner sep=0.75pt]    {$I_{j}$};
    	\draw (116,146.26) node [anchor=north west][inner sep=0.75pt]    {$B_{j,t}$};
	\draw (125,85) node {$I_{j,t}^+$};
	\draw[->] (130,80) -- (145,74);
	\draw (127,335) node {$I_{j,t}^-$};
	\draw[->] (132,343) -- (143,352);
	\end{tikzpicture}
    \caption{The three types of disks in $\Sigma_\alpha$, and intersection arcs.}
    \label{Figure:Nbhds of sphere}
\end{figure}
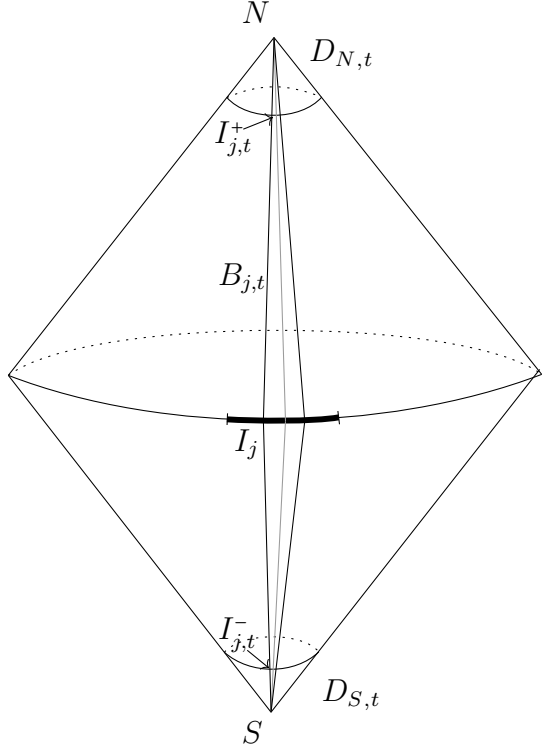

In addition, let $I_j^\circ$ denote the interior of $I_j$, and we define $B_j = \pi(I_j^\circ \times [-1,1])$, which is an open disk union with $\{N,S\}$, since $\pi_\alpha$ is injective on $I_j \times (-1,1)$.  With this definition, we observe that $\displaystyle{\bigcup_{t \in (0,1)} B_{j,t} = B_j}$, for all $j \in \mathbb Z$.  We identify each of these subspaces of $\Sigma_\alpha = \mathbb S_\alpha \setminus C_\alpha$ as subspaces of $M_{F_\alpha}$ via the homeomorphism $\Sigma_\alpha  \cong \Sigma_\alpha \times\{0\}$ composed with the composition with quotient map $\pi_M \colon \Sigma_\alpha \times [0,1] \to M_{F_\alpha}$.

It will also be useful to write $(\psi_s \colon M_{F_\alpha} \to M_{F_\alpha})_{s \in \mathbb R}$ to denote the suspension flow (the descent of the local flow on $\Sigma_\alpha \times [0,1]$ given by $(x,u) \mapsto (x,u+s)$, where defined).  The first return map of $(\psi_s)$ to $\Sigma_\alpha$ is $F_\alpha$.

As indicated above, our exhaustion of $M_{F_\alpha}$ will be defined by thickening the exhaustion of $M_F$.
First, the two circles, $\bar e_\nu,\bar e_\sigma$, in the cell structure on $M_F$ can be thickened in $M_{F_\alpha}$ to solid tori of varying thickness, depending on $t \in (0,1)$.  Explicitly, for each $t \in (0,1)$, set
\[ \mathcal D_t = \pi_M((D_{N,t} \cup D_{S,t}) \times [0,1]).\]
Since $D_{N,t}$ and $D_{S,t}$ are invariant by $F_\alpha$, it follows that $\mathcal D_t$ is $\psi_s$--invariant, for all $t$.  As the mapping torus of a pair of disks by a homeomorphism isotopic to the identity,  $\mathcal D_t$ is homeomorphic to a disjoint union of two solid tori having $\bar e_\nu \cup \bar e_\sigma$ as core curves, for all $t$.  Additionally, for all $t < t'$, $\mathcal D_t \subset \mathcal D_{t'}$ and $\overline{\mathcal D_{t'} - \mathcal D_t} \cong T^2 \times [t,t']$.

To thicken the rectangles in the infinite strip in $M_F$, we proceed as follows.
First, define
\[ \Psi \colon B_0 \times \mathbb R \to M_{F_\alpha}\]
by $\Psi(x,u) = \psi_u(x)$. We illustrate this map in Figure~\ref{fig:susp-flow}.

\begin{figure}[ht]
    \centering

\tikzset{every picture/.style={line width=0.75pt}} 

\begin{tikzpicture}[x=0.75pt,y=0.75pt,yscale=-1,xscale=1]

\draw    (111.9,83.51) .. controls (127.23,112.81) and (176.87,105.33) .. (183.62,83.51) ;
\draw    (120.85,93.92) .. controls (132.62,83.95) and (158.61,79.46) .. (175.53,94.17) ;

\draw   (0,93.25) .. controls (0,41.75) and (65.92,0) .. (147.23,0) .. controls (228.54,0) and (294.46,41.75) .. (294.46,93.25) .. controls (294.46,144.75) and (228.54,186.5) .. (147.23,186.5) .. controls (65.92,186.5) and (0,144.75) .. (0,93.25) -- cycle ;
\draw [color={rgb, 255:red, 0; green, 0; blue, 0 }  ,draw opacity=1 ]   (147.36,103.14) .. controls (125.82,115.96) and (127.4,180.09) .. (147.23,186.5) ;
\draw  [dash pattern={on 0.84pt off 2.51pt}]  (147.36,103.14) .. controls (168.38,109.55) and (165.22,180.62) .. (147.23,186.5) ;
\draw [color={rgb, 255:red, 155; green, 155; blue, 155 }  ,draw opacity=1 ][line width=2.5]    (135.88,119.56) .. controls (133.44,128.3) and (132.02,135.08) .. (131.89,141.51) ;
\draw [color={rgb, 255:red, 155; green, 155; blue, 155 }  ,draw opacity=1 ][line width=2.5]    (132.39,156.77) .. controls (133.68,167.06) and (133.49,164.82) .. (135.7,172.3) ;
\draw [color={rgb, 255:red, 144; green, 19; blue, 254 }  ,draw opacity=1 ]   (114.23,133.63) .. controls (168.01,130.47) and (239.96,125.8) .. (240.87,92.36) .. controls (241.79,58.93) and (201.11,16.38) .. (146.41,17.78) .. controls (91.71,19.19) and (33.33,39.76) .. (30.8,92.6) .. controls (28.27,145.44) and (126.18,169.98) .. (149.4,163.2) ;

\draw (118.11,111.43) node [anchor=north west][inner sep=0.75pt] {$I_{0}$};
\draw (116.97,167.54) node [anchor=north west][inner sep=0.75pt] {$I_{1}$};
\draw (284.07,147.24) node [anchor=north west][inner sep=0.75pt] {$\partial D_{t}$};

\end{tikzpicture}

    \caption{The map $\Psi$ restricted to one of the connected components of $\mathcal D_t$. We depict the orbit of the point $y_{j,t}:=\pi_M(y_j,(1-t))$.}
    \label{fig:susp-flow}
\end{figure}
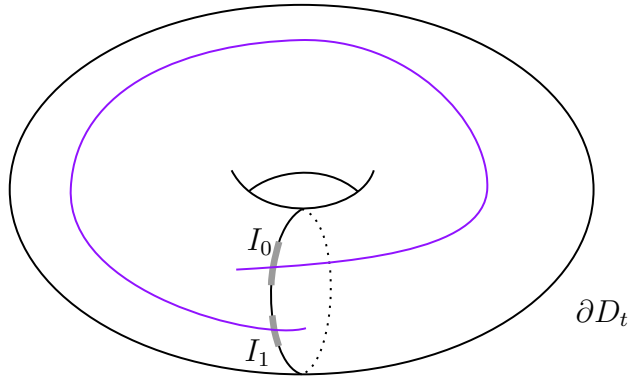

For each $t \in (0,1)$, $\Psi( B_{0,t} \times \mathbb R)$ is a thickening of the strip. Indeed, for all $j \in \mathbb Z$, $\Psi(\bar e_0 \times [j,j+1]))$ is precisely the $2$--cell bounded by $\bar e_j,\bar e_{j+1},\bar e_{\nu}, \bar e_{\sigma}$.
Moreover, the restriction of $\Psi$ to $(B_0 \setminus \{N,S\}) \times \mathbb R$ is injective, and the image defines a product neighborhood of the interior of the strip, while the complement $\Psi(\{N,S\} \times \mathbb R)$ maps to the core curves $\bar e_\nu \cup \bar e_\sigma$ of the solid tori $\mathcal D_t$, for all $t$.
The image $\Psi(B_0 \times \mathbb R)$ is $\psi_s$--invariant, as is $\Psi(B_{0,t} \times \mathbb R)$ for all $t \in (0,1)$. Moreover, for all $t < t'$,
\[ \Psi(B_{0,t} \times \mathbb R) \subset \Psi(B_{0,t'} \times \mathbb R).\]

Next we choose specific $1$--parameter family of rectangles in the exhaustion of the strip which we will thicken.  For this, let $\tau \colon (0,1) \to (0,\infty)$ be any orientation preserving homeomorphism and then set
\[ W_t = \mathcal D_t \cup \Psi\left(B_{0,t} \times [-\tau(t),\tau(t)] \right).\]

\medskip


\noindent
{\bf Claim 1.} $W_t \cong V_2$ for all $t \in (0,1)$.
\begin{proof} Observe that $W_t$ is the union of the two solid tori, $\mathcal D_t$, together with
\[ \Delta_t = \Psi(B_{0,t} \setminus (D_{N,t}^\circ \cup D_{S,t}^\circ) \times [-\tau(t),\tau(t)]).\]
Next, note that $B_{0,t} \setminus (D_{N,t}^\circ \cup D_{S,t}^\circ)$ is homemorphic to a disk.  Since $\{N,S\} \subset D_{N,t} \cup D_{S,t}$, it follows that the restriction of $\Psi$ to
\[ (B_{0,t} \setminus (D_{N,t}^\circ \cup D_{S,t}^\circ)) \times [-\tau(t),\tau(t)]\]
is injective, and hence $\Delta_t$ is a $3$--ball.
Moreover, this $3$--ball meets $\mathcal D_t$ in two $2$--disks,
\[ \Psi(I_{0,t}^\pm \times [-\tau(t),\tau(t)]).\]
That is, $W_t$ is obtained from $\mathcal D_t$ by attaching a single $1$--handle along two $2$--disks, one on each component of $\mathcal D_t$.  It follows that $W_t$ is a genus $2$ handlebody, as required.
\end{proof}

\noindent
{\bf Claim 2.} $\overline{W_{t'} \setminus W_t} \cong S_2 \times [t,t']$ for all $t < t'$.
\begin{proof} We can view $\overline{W_{t'} \setminus W_t}$ as obtained from $W_t$ by first expanding $\mathcal D_t$ to $\mathcal D_{t'}$, then expanding $\Phi(B_{0,t} \times [-\tau(t),\tau(t)]$ to $\Phi(B_{0,t'} \times [-\tau(t'),\tau(t')]$.   This first expansion has the effect of adding a collar neighborhood to two $1$--holed tori in $\partial W_t$, and the second expansion adds a collar neighborhood to an annulus.  Together the result is homeomorphic to the required product. See Figure~\ref{Fig:product}.
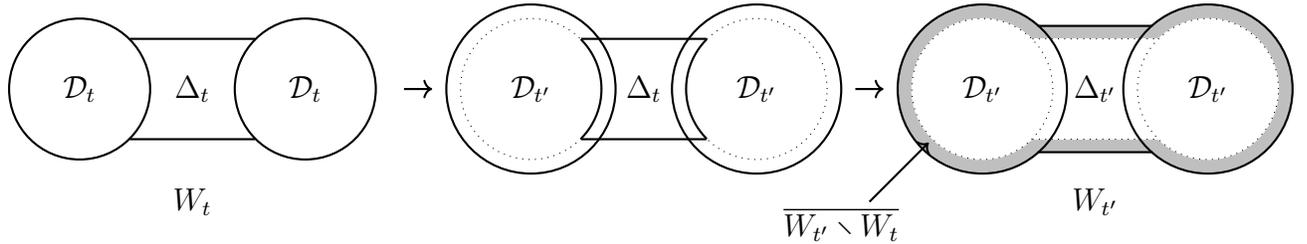
\begin{figure}[h]

\begin{tikzpicture}[scale = .75]
\draw[thick] (0,0) circle (1.25);
\draw[thick] (4,0) circle (1.25);
\draw[thick] (.89,.89) -- (3.11,.89);
\draw[thick] (.89,-.89) -- (3.11,-.89);
\node at (0,0) {$\mathcal D_t$};
\node at (4,0) {$\mathcal D_t$};
\node at (2,0) {$\Delta_t$};
\node at (2,-2) {$W_t$};
\draw[thick,->] (5.75,0) -- (6.25,0);
\draw[dotted] (8,0) circle (1.25);
\draw[dotted] (12,0) circle (1.25);
\draw[thick] (8,0) circle (1.5);
\draw[thick] (12,0) circle (1.5);
\draw[thick] (8.89,.89) -- (11.11,.89);
\draw[thick] (8.89,-.89) -- (11.11,-.89);
\node at (8,0) {$\mathcal D_{t'}$};
\node at (12,0) {$\mathcal D_{t'}$};
\node at (10,0) {$\Delta_t$};
\draw[thick,->] (13.75,0) -- (14.25,0);
\draw[thick] [domain=-45:45] plot ({8+1.25*cos(\x)}, {1.25*sin(\x)});
\draw[thick] [domain=-45:45] plot ({12-1.25*cos(\x)}, {1.25*sin(\x)});
\filldraw[lightgray] (16,0) circle (1.5);
\filldraw[lightgray] (20,0) circle (1.5);
\filldraw[lightgray] (16.98,1.12) -- (19.02,1.12) -- (19.02,.89) -- (16.98,.89) (16.98,1.12);
\filldraw[lightgray] (16.98,-1.12) -- (19.02,-1.12) -- (19.02,-.89) -- (16.98,-.89) (16.98,-1.12);
\filldraw[white] (16,.89) -- (20,.89) -- (20,-.89) -- (16,-.89) -- (16,.89);
\filldraw[white] (16,0) circle (1.25);
\filldraw[white] (20,0) circle (1.25);
\draw[dotted] (16,0) circle (1.25);
\draw[dotted] (20,0) circle (1.25);
\draw[thick] (16,0) circle (1.5);
\draw[thick] (20,0) circle (1.5);
\draw[dotted] (16.89,.89) -- (19.11,.89);
\draw[dotted] (16.89,-.89) -- (19.11,-.89);
\draw[thick] (16.98,1.12) -- (19.02,1.12);
\draw[thick] (16.98,-1.12) -- (19.02,-1.12);
\node at (16,0) {$\mathcal D_{t'}$};
\node at (20,0) {$\mathcal D_{t'}$};
\node at (18,0) {$\Delta_{t'}$};
\node at (18,-2) {$W_{t'}$};
\draw[thick,->] (14,-2) -- (15.05,-.95);
\node at (13.5,-2.4) {$\overline{W_{t'} \setminus W_t}$};
\end{tikzpicture}
\caption{A schematic diagram describing how to enlarge $W_t$ to $W_{t'}$ in two steps.}
\label{Fig:product}
\end{figure}
\end{proof}
The two claims complete the proof of Proposition~\ref{Prop:mapping torus handlebody} .
\end{proof}


\subsection{New constructions}
	\label{ssec:new-constructions}
We now begin describing variations on the construction above, starting with the simplest such variation.
Consider a homeomorphism $f_{\alpha,n}$ which results from blowing-up $n\geq 2$ orbits of the irrational rotation $\rho_\alpha$, each orbit being blown-up as detailed in Section~\ref{section:Denjoy}. By the same arguments detailed in that section, $f_{\alpha,n}$ can be extended to a map $F_{\alpha,n}\in\Homeo(\Sigma_{\alpha,n})$, where $\Sigma_{\alpha,n}:=\mathbb{S}_\alpha^2\setminus C_{\alpha,n}$ and $C_{\alpha,n}$ is homeomorphic to a Cantor set.
Analyzing the proof of Proposition \ref{Prop:mapping torus handlebody}, the blown-up orbit resulted in the $1$--handle joining two solid tori in $M_{F_{\alpha}}$. Following the same argumentation, blowing-up $n$ orbits produces $n$ $1$-handles joining two solid tori in $M_{F_{\alpha,n}}$. This way, we obtain the following:

\begin{proposition}
	\label{prop:blup-n-orbits}
The mapping torus of $F_{\alpha,n}$ is homeomorphic to $V_{n+1}$ and defines a fibration of $\{V_{n+1}\xrightarrow{p_{\alpha,n}}\mathbb{S}^1\}$ where every fiber is homeomorphic to the Cantor tree surface $\Sigma_0$.
\end{proposition}


\begin{remark}
	\label{rmk:taking-powers}
One could also consider, for each $n\geq 2$, the $n^{th}$-power $f_\alpha^n$ of the Denjoy homeomorphism $f_\alpha$. The map $f_\alpha^n$ leaves also the Cantor set $C_\alpha$ invariant, but it has $n$ different wandering intervals, and extends to $F_\alpha^n$ in $\Homeo(\Sigma_\alpha)$.  The mapping torus $M_{F_\alpha^n}$ is homeomorphic to $V_{n+1}$ and defines a fibration $\{V_{k+1}\xrightarrow{p_{\alpha,k}}\mathbb{S}^1\}$ where every fiber is homeomorphic to the Cantor tree surface $\Sigma_0$; one can see this as an application of Proposition~\ref{prop:blup-n-orbits}, but it also follows from the fact that $M_{F_\alpha^n}$ is an $n$--fold cover of $M_{F_\alpha} \cong V_2$.
\end{remark}


In the next section, we describe some more elaborate variations on the construction from this section.

\section{Other examples}
	\label{section:examples}

In this section we use the homeomorphism $F_\alpha$ to construct ``building blocks" that can be used to produce homeomorphisms of other infinite type surfaces whose mapping tori are tame.  We describe a few such constructions, but there are numerous others.  Before we describe the specific constructions, we make some basic topological observations about gluing together and decomposing mapping tori in terms of gluing together and decomposing the fibers and their monodromies.

Suppose $g_i \colon Z_i \to Z_i$ are homeomorphisms of surfaces with non-empty boundary $Z_i$, for $i=1,2$.  Further suppose that there exists a homeomorphism $\varphi \colon \partial Z_1 \to \partial Z_2$ conjugating the restriction $g_1|_{\partial Z_1}$ to $g_2|_{\partial Z_2}$.  We may glue $Z_1$ and $Z_2$ together along the boundaries by $\varphi$, $Z = Z_1 \cup_\varphi Z_2$, and the homeomorphisms $g_1,g_2$ glue together to produce a homeomorphism $g \colon Z \to Z$ so that $g|_{Z_i} = g_i$, for $i=1,2$.  The inclusions $Z_i \to Z$ determine inclusions $Z_i \times [0,1] \to Z \times [0,1]$, for $i=1,2$.  These descend to embeddings
\[ M_{g_i} \to M_g,\]
for $i=1,2$, and the images meet precisely along their boundaries.

Equivalently, we can view $M_g$ as obtained by gluing together $M_{g_1}$ and $M_{g_2}$ along their boundaries.  Explicitly, the homeomorphism $\varphi \times \mbox{id}_{[0,1]} \colon (\partial Z_1) \times [0,1] \to (\partial Z_2) \times [0,1]$ descends to a homeomorphism $\bar \varphi \colon \partial M_{g_1} \to \partial M_{g_2}$, and $M_g$ is obtained by gluing $M_{g_1}$ to $M_{g_2}$ via this homeomorphism,
\[ M_g = M_{g_1} \cup_{\bar \varphi} M_{g_2}.\]

We will also want to run this construction ``in reverse", in the following sense.  Suppose $F \colon \Sigma \to \Sigma$ is a homeomorphism and $Y \subset \Sigma$ is a properly embedded subsurface such that $F(Y) = Y$.  The proper embedding assumption implies that $Y^c := \Sigma - int(Y)$ is also a properly embedded subsurface, and we set $F_Y = F|_Y$, and $F_{Y^c} = F|_{Y^c}$.  Then the mapping torus $M_{F_Y}$ is given by,
\[ M_{F_Y} = M_F \setminus int(M_{F_{Y^c}}).\]

In the next two sections, we describe two particular instances of these constructions.  We will make explicit use of the terminology and notation from \S\ref{section:proofs} in what follows.

\subsection{Building blocks: Disk and annulus minus Cantor set}

First, recall that in the exhaustion of $M_{F_\alpha}$, there are the two solid tori, $\mathcal D_{1/2} \subset M_{F_\alpha}$.  These solid tori are mapping tori of $F_\alpha$--invariant disks $D_{N,1/2} \cup D_{S,1/2}$, or more precisely the restriction of $F_\alpha$ to the union of these disks.  In fact, each of $D_{N,1/2}$ and $D_{S,1/2}$ are invariant by $F_\alpha$, and we write
\[ \mathcal D_{1/2} = \mathcal D_{N,1/2} \sqcup \mathcal D_{S,1/2},\]
where $\mathcal D_{N,1/2}$ and $\mathcal D_{S,1/2}$ are each solid tori, obtained as the mapping tori of $F_\alpha$ restricted to each of $D_{N,1/2}$ and $D_{S,1/2}$, respectively.

Now, set $Y^1_\alpha = \Sigma_\alpha \setminus int(D_{N,1/2})$.  Since $Y^1_\alpha$ is also invariant by $F_\alpha$, we can define
\[ h^1_\alpha = (F_\alpha)|_{Y^1_\alpha} \colon Y^1_\alpha \to Y^1_\alpha.\]
As described above, the mapping torus of $h^1_\alpha$ is obtained from $M_{F_\alpha}$ by removing the mapping torus of the restriction of $F_\alpha$ to $D_{N,1/2}$,
\[ M_{h^1_\alpha} = M_{F_\alpha} \setminus int(\mathcal D_{N,1/2}).\]
That is, $M_{h^1_\alpha}$ is obtained by removing the open solid torus $int(\mathcal D_{N,1/2})$ from the open handelbody $M_{F_\alpha} \cong V_2$, resulting in the $3$--manifold with non-empty boundary shown in Figure~\ref{fig:MThalpha1}; the boundary consists precisely of the torus ``inside" the handlebody (the outer boundary of the handlebody is still missing).

If instead we set $Y^2_\alpha = \Sigma \setminus int(D_{N,1/2} \cup D_{S,1/2})$ and $h^2_\alpha = (F_\alpha)|_{Y^2_\alpha}$, then $M_{h^2_\alpha}$ is homeomorphic to $M_{F_\alpha}$ minus the interior of the union of the two solid tori, $M_{h^2_\alpha} = M_{F_\alpha} \setminus int(\mathcal D_{1/2})$ as on Figure~\ref{fig:MThalpha2}.  Again, the boundary consists of two tori.

 Each of $M_{h_\alpha^1}$ and $M_{h_\alpha^2}$ is obtained by removing a boundary components from a {\em compression body}, which is a $3$--manifold obtained by attaching $1$--handles to a product of a surface with an interval; see \cite[Section~2.2]{Scharlemann}.   Explicitly, we can start with $T^2 \times [0,1]$ and attach a $1$--handle along a pair of disks in $T^2 \times \{1\}$.  The resulting manifold is a compression body with one genus $1$ boundary component and one genus $2$ boundary component, and then $M_{h_\alpha^1}$ is obtained by deleting the genus $2$ boundary component.  Similarly, we could start with a disjoint union of two tori,
\[ T^2 \times \{1,2\} = T^2_1 \sqcup T^2_2 \]
and take the product with $[0,1]$ to produce a disconnected $3$--manifold,
\[ (T^2_1 \sqcup T^2_2) \times [0,1]. \]
Attaching a $1$--handle attached along a pair of disks, one in $T^2_1 \times \{1\}$ and the other in $T^2_2 \times \{1\}$, results in a $3$--manifold with two genus $1$ boundary components and one genus $2$ boundary component; $M_{h_\alpha^2}$ is obtained from this by deleted the genus $2$ boundary component.



\begin{figure}[ht]
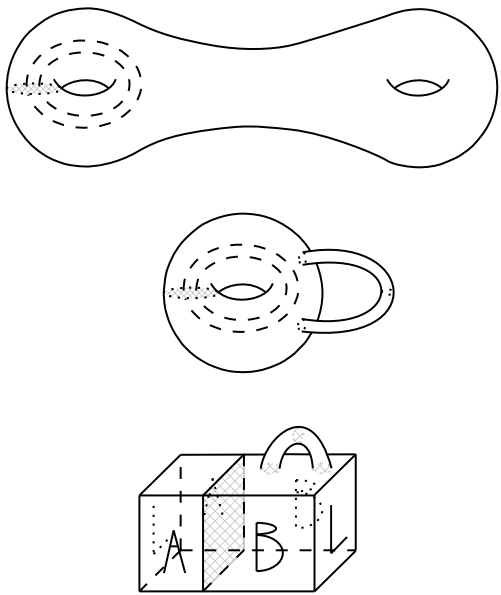



\tikzset{
pattern size/.store in=\mcSize,
pattern size = 5pt,
pattern thickness/.store in=\mcThickness,
pattern thickness = 0.3pt,
pattern radius/.store in=\mcRadius,
pattern radius = 1pt}
\makeatletter
\pgfutil@ifundefined{pgf@pattern@name@_4ygwwc2vs}{
\pgfdeclarepatternformonly[\mcThickness,\mcSize]{_4ygwwc2vs}
{\pgfqpoint{0pt}{0pt}}
{\pgfpoint{\mcSize}{\mcSize}}
{\pgfpoint{\mcSize}{\mcSize}}
{
\pgfsetcolor{\tikz@pattern@color}
\pgfsetlinewidth{\mcThickness}
\pgfpathmoveto{\pgfqpoint{0pt}{\mcSize}}
\pgfpathlineto{\pgfpoint{\mcSize+\mcThickness}{-\mcThickness}}
\pgfpathmoveto{\pgfqpoint{0pt}{0pt}}
\pgfpathlineto{\pgfpoint{\mcSize+\mcThickness}{\mcSize+\mcThickness}}
\pgfusepath{stroke}
}}
\makeatother


\tikzset{
pattern size/.store in=\mcSize,
pattern size = 5pt,
pattern thickness/.store in=\mcThickness,
pattern thickness = 0.3pt,
pattern radius/.store in=\mcRadius,
pattern radius = 1pt}
\makeatletter
\pgfutil@ifundefined{pgf@pattern@name@_f3gvhf3t5}{
\pgfdeclarepatternformonly[\mcThickness,\mcSize]{_f3gvhf3t5}
{\pgfqpoint{0pt}{0pt}}
{\pgfpoint{\mcSize}{\mcSize}}
{\pgfpoint{\mcSize}{\mcSize}}
{
\pgfsetcolor{\tikz@pattern@color}
\pgfsetlinewidth{\mcThickness}
\pgfpathmoveto{\pgfqpoint{0pt}{\mcSize}}
\pgfpathlineto{\pgfpoint{\mcSize+\mcThickness}{-\mcThickness}}
\pgfpathmoveto{\pgfqpoint{0pt}{0pt}}
\pgfpathlineto{\pgfpoint{\mcSize+\mcThickness}{\mcSize+\mcThickness}}
\pgfusepath{stroke}
}}
\makeatother


\tikzset{
pattern size/.store in=\mcSize,
pattern size = 5pt,
pattern thickness/.store in=\mcThickness,
pattern thickness = 0.3pt,
pattern radius/.store in=\mcRadius,
pattern radius = 1pt}
\makeatletter
\pgfutil@ifundefined{pgf@pattern@name@_e03f30mni}{
\pgfdeclarepatternformonly[\mcThickness,\mcSize]{_e03f30mni}
{\pgfqpoint{0pt}{0pt}}
{\pgfpoint{\mcSize}{\mcSize}}
{\pgfpoint{\mcSize}{\mcSize}}
{
\pgfsetcolor{\tikz@pattern@color}
\pgfsetlinewidth{\mcThickness}
\pgfpathmoveto{\pgfqpoint{0pt}{\mcSize}}
\pgfpathlineto{\pgfpoint{\mcSize+\mcThickness}{-\mcThickness}}
\pgfpathmoveto{\pgfqpoint{0pt}{0pt}}
\pgfpathlineto{\pgfpoint{\mcSize+\mcThickness}{\mcSize+\mcThickness}}
\pgfusepath{stroke}
}}
\makeatother


\tikzset{
pattern size/.store in=\mcSize,
pattern size = 5pt,
pattern thickness/.store in=\mcThickness,
pattern thickness = 0.3pt,
pattern radius/.store in=\mcRadius,
pattern radius = 1pt}
\makeatletter
\pgfutil@ifundefined{pgf@pattern@name@_cy45yo1gf}{
\pgfdeclarepatternformonly[\mcThickness,\mcSize]{_cy45yo1gf}
{\pgfqpoint{0pt}{0pt}}
{\pgfpoint{\mcSize}{\mcSize}}
{\pgfpoint{\mcSize}{\mcSize}}
{
\pgfsetcolor{\tikz@pattern@color}
\pgfsetlinewidth{\mcThickness}
\pgfpathmoveto{\pgfqpoint{0pt}{\mcSize}}
\pgfpathlineto{\pgfpoint{\mcSize+\mcThickness}{-\mcThickness}}
\pgfpathmoveto{\pgfqpoint{0pt}{0pt}}
\pgfpathlineto{\pgfpoint{\mcSize+\mcThickness}{\mcSize+\mcThickness}}
\pgfusepath{stroke}
}}
\makeatother


\tikzset{
pattern size/.store in=\mcSize,
pattern size = 5pt,
pattern thickness/.store in=\mcThickness,
pattern thickness = 0.3pt,
pattern radius/.store in=\mcRadius,
pattern radius = 1pt}
\makeatletter
\pgfutil@ifundefined{pgf@pattern@name@_4mqkm09zd}{
\pgfdeclarepatternformonly[\mcThickness,\mcSize]{_4mqkm09zd}
{\pgfqpoint{0pt}{0pt}}
{\pgfpoint{\mcSize}{\mcSize}}
{\pgfpoint{\mcSize}{\mcSize}}
{
\pgfsetcolor{\tikz@pattern@color}
\pgfsetlinewidth{\mcThickness}
\pgfpathmoveto{\pgfqpoint{0pt}{\mcSize}}
\pgfpathlineto{\pgfpoint{\mcSize+\mcThickness}{-\mcThickness}}
\pgfpathmoveto{\pgfqpoint{0pt}{0pt}}
\pgfpathlineto{\pgfpoint{\mcSize+\mcThickness}{\mcSize+\mcThickness}}
\pgfusepath{stroke}
}}
\makeatother


\tikzset{
pattern size/.store in=\mcSize,
pattern size = 5pt,
pattern thickness/.store in=\mcThickness,
pattern thickness = 0.3pt,
pattern radius/.store in=\mcRadius,
pattern radius = 1pt}
\makeatletter
\pgfutil@ifundefined{pgf@pattern@name@_vfrzxrixf}{
\pgfdeclarepatternformonly[\mcThickness,\mcSize]{_vfrzxrixf}
{\pgfqpoint{0pt}{0pt}}
{\pgfpoint{\mcSize}{\mcSize}}
{\pgfpoint{\mcSize}{\mcSize}}
{
\pgfsetcolor{\tikz@pattern@color}
\pgfsetlinewidth{\mcThickness}
\pgfpathmoveto{\pgfqpoint{0pt}{\mcSize}}
\pgfpathlineto{\pgfpoint{\mcSize+\mcThickness}{-\mcThickness}}
\pgfpathmoveto{\pgfqpoint{0pt}{0pt}}
\pgfpathlineto{\pgfpoint{\mcSize+\mcThickness}{\mcSize+\mcThickness}}
\pgfusepath{stroke}
}}
\makeatother
\tikzset{every picture/.style={line width=0.75pt}} 



\caption{Three figures illustrating homeomorphic instances of the mapping torus of $h^1_\alpha$, the restriction of $F_\alpha$ to $Y^1_\alpha = \Sigma_\alpha - int(D_{N,1/2})$.}\label{fig:MThalpha1}
\end{figure}

\begin{figure}[ht]
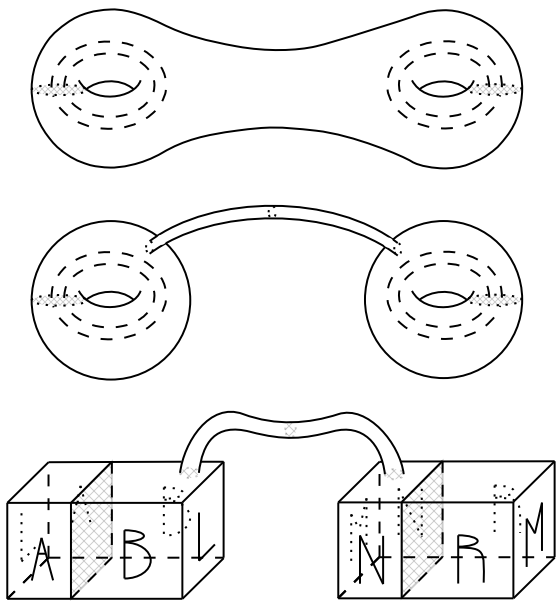



\tikzset{
pattern size/.store in=\mcSize,
pattern size = 5pt,
pattern thickness/.store in=\mcThickness,
pattern thickness = 0.3pt,
pattern radius/.store in=\mcRadius,
pattern radius = 1pt}
\makeatletter
\pgfutil@ifundefined{pgf@pattern@name@_hlitene11}{
\pgfdeclarepatternformonly[\mcThickness,\mcSize]{_hlitene11}
{\pgfqpoint{0pt}{0pt}}
{\pgfpoint{\mcSize}{\mcSize}}
{\pgfpoint{\mcSize}{\mcSize}}
{
\pgfsetcolor{\tikz@pattern@color}
\pgfsetlinewidth{\mcThickness}
\pgfpathmoveto{\pgfqpoint{0pt}{\mcSize}}
\pgfpathlineto{\pgfpoint{\mcSize+\mcThickness}{-\mcThickness}}
\pgfpathmoveto{\pgfqpoint{0pt}{0pt}}
\pgfpathlineto{\pgfpoint{\mcSize+\mcThickness}{\mcSize+\mcThickness}}
\pgfusepath{stroke}
}}
\makeatother


\tikzset{
pattern size/.store in=\mcSize,
pattern size = 5pt,
pattern thickness/.store in=\mcThickness,
pattern thickness = 0.3pt,
pattern radius/.store in=\mcRadius,
pattern radius = 1pt}
\makeatletter
\pgfutil@ifundefined{pgf@pattern@name@_vjy3gwiik}{
\pgfdeclarepatternformonly[\mcThickness,\mcSize]{_vjy3gwiik}
{\pgfqpoint{0pt}{0pt}}
{\pgfpoint{\mcSize}{\mcSize}}
{\pgfpoint{\mcSize}{\mcSize}}
{
\pgfsetcolor{\tikz@pattern@color}
\pgfsetlinewidth{\mcThickness}
\pgfpathmoveto{\pgfqpoint{0pt}{\mcSize}}
\pgfpathlineto{\pgfpoint{\mcSize+\mcThickness}{-\mcThickness}}
\pgfpathmoveto{\pgfqpoint{0pt}{0pt}}
\pgfpathlineto{\pgfpoint{\mcSize+\mcThickness}{\mcSize+\mcThickness}}
\pgfusepath{stroke}
}}
\makeatother


\tikzset{
pattern size/.store in=\mcSize,
pattern size = 5pt,
pattern thickness/.store in=\mcThickness,
pattern thickness = 0.3pt,
pattern radius/.store in=\mcRadius,
pattern radius = 1pt}
\makeatletter
\pgfutil@ifundefined{pgf@pattern@name@_alt0oe3wz}{
\pgfdeclarepatternformonly[\mcThickness,\mcSize]{_alt0oe3wz}
{\pgfqpoint{0pt}{0pt}}
{\pgfpoint{\mcSize}{\mcSize}}
{\pgfpoint{\mcSize}{\mcSize}}
{
\pgfsetcolor{\tikz@pattern@color}
\pgfsetlinewidth{\mcThickness}
\pgfpathmoveto{\pgfqpoint{0pt}{\mcSize}}
\pgfpathlineto{\pgfpoint{\mcSize+\mcThickness}{-\mcThickness}}
\pgfpathmoveto{\pgfqpoint{0pt}{0pt}}
\pgfpathlineto{\pgfpoint{\mcSize+\mcThickness}{\mcSize+\mcThickness}}
\pgfusepath{stroke}
}}
\makeatother


\tikzset{
pattern size/.store in=\mcSize,
pattern size = 5pt,
pattern thickness/.store in=\mcThickness,
pattern thickness = 0.3pt,
pattern radius/.store in=\mcRadius,
pattern radius = 1pt}
\makeatletter
\pgfutil@ifundefined{pgf@pattern@name@_qlgng2g6o}{
\pgfdeclarepatternformonly[\mcThickness,\mcSize]{_qlgng2g6o}
{\pgfqpoint{0pt}{0pt}}
{\pgfpoint{\mcSize}{\mcSize}}
{\pgfpoint{\mcSize}{\mcSize}}
{
\pgfsetcolor{\tikz@pattern@color}
\pgfsetlinewidth{\mcThickness}
\pgfpathmoveto{\pgfqpoint{0pt}{\mcSize}}
\pgfpathlineto{\pgfpoint{\mcSize+\mcThickness}{-\mcThickness}}
\pgfpathmoveto{\pgfqpoint{0pt}{0pt}}
\pgfpathlineto{\pgfpoint{\mcSize+\mcThickness}{\mcSize+\mcThickness}}
\pgfusepath{stroke}
}}
\makeatother


\tikzset{
pattern size/.store in=\mcSize,
pattern size = 5pt,
pattern thickness/.store in=\mcThickness,
pattern thickness = 0.3pt,
pattern radius/.store in=\mcRadius,
pattern radius = 1pt}
\makeatletter
\pgfutil@ifundefined{pgf@pattern@name@_jtnssmqlm}{
\pgfdeclarepatternformonly[\mcThickness,\mcSize]{_jtnssmqlm}
{\pgfqpoint{0pt}{0pt}}
{\pgfpoint{\mcSize}{\mcSize}}
{\pgfpoint{\mcSize}{\mcSize}}
{
\pgfsetcolor{\tikz@pattern@color}
\pgfsetlinewidth{\mcThickness}
\pgfpathmoveto{\pgfqpoint{0pt}{\mcSize}}
\pgfpathlineto{\pgfpoint{\mcSize+\mcThickness}{-\mcThickness}}
\pgfpathmoveto{\pgfqpoint{0pt}{0pt}}
\pgfpathlineto{\pgfpoint{\mcSize+\mcThickness}{\mcSize+\mcThickness}}
\pgfusepath{stroke}
}}
\makeatother


\tikzset{
pattern size/.store in=\mcSize,
pattern size = 5pt,
pattern thickness/.store in=\mcThickness,
pattern thickness = 0.3pt,
pattern radius/.store in=\mcRadius,
pattern radius = 1pt}
\makeatletter
\pgfutil@ifundefined{pgf@pattern@name@_apll05on5}{
\pgfdeclarepatternformonly[\mcThickness,\mcSize]{_apll05on5}
{\pgfqpoint{0pt}{0pt}}
{\pgfpoint{\mcSize}{\mcSize}}
{\pgfpoint{\mcSize}{\mcSize}}
{
\pgfsetcolor{\tikz@pattern@color}
\pgfsetlinewidth{\mcThickness}
\pgfpathmoveto{\pgfqpoint{0pt}{\mcSize}}
\pgfpathlineto{\pgfpoint{\mcSize+\mcThickness}{-\mcThickness}}
\pgfpathmoveto{\pgfqpoint{0pt}{0pt}}
\pgfpathlineto{\pgfpoint{\mcSize+\mcThickness}{\mcSize+\mcThickness}}
\pgfusepath{stroke}
}}
\makeatother


\tikzset{
pattern size/.store in=\mcSize,
pattern size = 5pt,
pattern thickness/.store in=\mcThickness,
pattern thickness = 0.3pt,
pattern radius/.store in=\mcRadius,
pattern radius = 1pt}
\makeatletter
\pgfutil@ifundefined{pgf@pattern@name@_rai70uakz}{
\pgfdeclarepatternformonly[\mcThickness,\mcSize]{_rai70uakz}
{\pgfqpoint{0pt}{0pt}}
{\pgfpoint{\mcSize}{\mcSize}}
{\pgfpoint{\mcSize}{\mcSize}}
{
\pgfsetcolor{\tikz@pattern@color}
\pgfsetlinewidth{\mcThickness}
\pgfpathmoveto{\pgfqpoint{0pt}{\mcSize}}
\pgfpathlineto{\pgfpoint{\mcSize+\mcThickness}{-\mcThickness}}
\pgfpathmoveto{\pgfqpoint{0pt}{0pt}}
\pgfpathlineto{\pgfpoint{\mcSize+\mcThickness}{\mcSize+\mcThickness}}
\pgfusepath{stroke}
}}
\makeatother


\tikzset{
pattern size/.store in=\mcSize,
pattern size = 5pt,
pattern thickness/.store in=\mcThickness,
pattern thickness = 0.3pt,
pattern radius/.store in=\mcRadius,
pattern radius = 1pt}
\makeatletter
\pgfutil@ifundefined{pgf@pattern@name@_a2ql54diu}{
\pgfdeclarepatternformonly[\mcThickness,\mcSize]{_a2ql54diu}
{\pgfqpoint{0pt}{0pt}}
{\pgfpoint{\mcSize}{\mcSize}}
{\pgfpoint{\mcSize}{\mcSize}}
{
\pgfsetcolor{\tikz@pattern@color}
\pgfsetlinewidth{\mcThickness}
\pgfpathmoveto{\pgfqpoint{0pt}{\mcSize}}
\pgfpathlineto{\pgfpoint{\mcSize+\mcThickness}{-\mcThickness}}
\pgfpathmoveto{\pgfqpoint{0pt}{0pt}}
\pgfpathlineto{\pgfpoint{\mcSize+\mcThickness}{\mcSize+\mcThickness}}
\pgfusepath{stroke}
}}
\makeatother


\tikzset{
pattern size/.store in=\mcSize,
pattern size = 5pt,
pattern thickness/.store in=\mcThickness,
pattern thickness = 0.3pt,
pattern radius/.store in=\mcRadius,
pattern radius = 1pt}
\makeatletter
\pgfutil@ifundefined{pgf@pattern@name@_v3sv8q07v}{
\pgfdeclarepatternformonly[\mcThickness,\mcSize]{_v3sv8q07v}
{\pgfqpoint{0pt}{0pt}}
{\pgfpoint{\mcSize}{\mcSize}}
{\pgfpoint{\mcSize}{\mcSize}}
{
\pgfsetcolor{\tikz@pattern@color}
\pgfsetlinewidth{\mcThickness}
\pgfpathmoveto{\pgfqpoint{0pt}{\mcSize}}
\pgfpathlineto{\pgfpoint{\mcSize+\mcThickness}{-\mcThickness}}
\pgfpathmoveto{\pgfqpoint{0pt}{0pt}}
\pgfpathlineto{\pgfpoint{\mcSize+\mcThickness}{\mcSize+\mcThickness}}
\pgfusepath{stroke}
}}
\makeatother
\tikzset{every picture/.style={line width=0.75pt}} 



    \caption{Three figures illustrating homeomorphic instances of the mapping torus of $h^2_\alpha$, $F_\alpha$ to $Y^2_\alpha = \Sigma_\alpha \setminus int(D_{N,1/2} \cup D_{S,1/2})$.}
    \label{fig:MThalpha2}
\end{figure}

We can apply an isotopy to $F_\alpha$, preserving $D_{N,1/2} \cup D_{S,1/2}$, so that the restriction to $D_{N,1/2} \cup D_{S,1/2}$ is the identity.  The restrictions $h^1_\alpha$ and $h^2_\alpha$ to $Y^1_\alpha$ and $Y^2_\alpha$ are isotopic to the restriction of the isotoped version of $F_\alpha$, so the resulting mapping tori are homeomorphic.  However, the new homeomorphisms of $Y^1_\alpha$ and $Y^2_\alpha$ are the identity on the boundary.  Replacing $h^1_\alpha$ and $h^2_\alpha$ with their isotoped versions, we have the following.

\begin{proposition}\label{Prop:discs and annuli building blocks} Suppose $\alpha \in (0,1) \setminus \mathbb Q$.  Then
\begin{enumerate}
\item there is a homeomorphism $h^1_\alpha$ of a disk minus a Cantor set $Y^1_\alpha$ such that $h^1_\alpha$ is the identity on the boundary and $M_{h^1_\alpha}$ is homeomorphic to $V_2 \setminus V_1$, where $V_1 \subset V_2$ is an unknotted core solid torus.
\item  there is a homeomorphism $h^2_\alpha$ of an annulus minus a Cantor set $Y^2_\alpha$ such that $h^2_\alpha$ is the identity on the boundary and $M_{h^2_\alpha}$ is homeomorphic to $V_2 \setminus V_1 \sqcup V_1'$ where $V_1 \sqcup V_1' \subset V_2$ is an unlinked union of core solid tori.
\end{enumerate}
Furthermore, the homeomorphisms $h^1_\alpha \colon \Ends(Y^1_\alpha) \to \Ends(Y^1_\alpha)$ and $h^2_\alpha \colon \Ends(Y^2_\alpha) \to \Ends(Y^2_\alpha)$ are conjugate to $f_\alpha \colon C_\alpha \to C_\alpha$.
\end{proposition}

Now for any surface with non-empty boundary $Z$ and homeomorphism $g \colon Z \to Z$ that is the identity on the boundary, we can glue on copies of $Y^1_\alpha$ to any circle boundary component of $Z$, and/or $Y^2_\alpha$ to any pair of circle boundary components of $Z$, and glue together $g$ with $h^1_\alpha$ and/or $h^2_\alpha$ on each copy of $Y^1_\alpha$ and $Y^2_\alpha$ to produce a new homeomorphism $\hat g$ on the glued surface $\hat Z$.  If $Z$ is compact, we claim that $M_{\hat g}$ is tame,  and up to homeomorphism is obtained from $M_g$ by a simple topological operation which we now describe.

First, suppose $\hat Z$ is obtained by gluing $Y_\alpha^1$ to $Z$ so that $\hat g$ is obtained by gluing $h_\alpha^1$ to $g$.  Then we obtain $M_{\hat g}$ by gluing $M_{h_\alpha^1}$ to the compact manifold-with-boundary, $M_g$ along a torus boundary component in each.  From the compression body description of $M_{h_\alpha^1} \cong V_2 - V_1$ above, we can alternatively describe $M_{\hat g}$ as obtained by first gluing $T^2 \times [0,1]$ to $M_g$, identifying $T^2 \times \{0\}$ to a torus boundary component of $M_g$, attaching a $1$--handle to the result, and then deleted the resulting genus $2$ boundary component.  Since gluing $T^2 \times [0,1]$ to $M_g$ produces a manifold homoemorphic to $M_g$, we can alternatively view $M_{\hat g}$ as simply obtained from $M_g$ by attaching a $1$--handle to a torus boundary component, then deleting the resulting genus $2$ boundary component.  In particular, $M_{\hat g}$ is indeed tame.


Similarly, for every copy of $Y^2_\alpha$ glued to $Z$, we glue a copy of $V_2 \setminus V_1 \sqcup V_1'$ to a pair of boundary tori.  Topologically, this is equivalent to adding a $1$--handle between two torus boundary components, then deleting the resulting genus $2$ boundary component.



The following Proposition, which provides a family of examples that can be derived from this construction, summarizes our previous discussion:

\begin{proposition}\label{Prop:gluing handles to 3-manifolds} For any compact, orientable 3--manifold that fibers over $\mathbb{S}^1$, with non-empty (torus) boundary with monodromy that is the identity on the boundary, then attaching a $1$--handle to every boundary torus, either connecting pairs of boundary components with a single $1$--handle or connecting a boundary torus to itself with a single $1$--handle, and then removing the resulting boundary, produces a tame $3$--manifold that fibers over the circle with fiber an infinite type surface with a Cantor set of ends.
\end{proposition}


The ideas leading to Proposition~\ref{Prop:gluing handles to 3-manifolds} can still be perturbed to provide a wider family of examples. Consider for example the case where the map $\hat{g}$ defining the mapping torus $M_{\hat{g}}$ is obtained by glueing $g$ with finitely many
maps $h^1_{\alpha,k}=h^1_\alpha$ and/or $h^2_{\alpha}=h^2_\alpha$ on copies of $Y^1_\alpha$ and $Y^2_\alpha$. One could consider now $\hat{g}'$ obtained by glueing $\hat{g}$ with finitely many $h^1_{\alpha_j}$ and/or $h^2_{\alpha_i}$ so that $\{\alpha_i,\alpha_j\}_{i\in I, j\in J}$ is a family of pairwise different elements of $(0,1)\setminus\mathbb{Q}$. The resulting mapping torus $M_{\hat{g}}$ would also be a tame fibered 3--manifold with fiber an infinite type surface with a Cantor set of ends.


\subsection{Building blocks: Sphere minus Cantor set and countable union of disks.}

For the next construction we will start again with $\Sigma_\alpha$ and this time remove a properly embedded subsurface which is a disjoint union of disks accumulating on the entire end space.  To get invariance of this subsurface, we will need to apply an isotopy of $F_\alpha$ first, however.  Before describing the isotopy, we explain roughly where the disks come from.

First, recall that $\{y_j\}_{j \in \mathbb Z}$ is an $F_\alpha$--invariant subset of the equator $\mathbb S_\alpha^1 \subset \mathbb S_\alpha^2$, consisting of points in the middle of each complementary interval of the invariant Cantor set $C_\alpha$.  Let
\[L = \bigcup_{s \in \R} \psi_s(y_0).\]
We claim that $L$ is an unknotted, properly embedded arc in $M_{F_\alpha}$.  To see this, we inspect the proof of Proposition~\ref{Prop:mapping torus handlebody}.  For all $t \in (0,1)$, we observe that the intersection $L \cap W_t$ of $L$ with the handlebody $W_t$ is contained in the $1$--handle, $\Delta_t$,
\[ L \cap W_t = L \cap \Delta_t \subset \Delta_t \cong D^2 \times [-1,1].\]
Moreover, we can choose a homeomorphism $\Delta_t \to D^2 \times [-1,1]$ sending $L \cap \Delta_t$ to an arc of the form $\delta \times \{0\}$, where $\delta \subset D^2$ is a diameter of the disk.  A properly embedded, closed tubular neighborhood of $L$ in $M_{F_\alpha}$ is then also unknotted, and removing its interior results in a space homeomorphic to the interior of $V_2$ minus the unknotted $int(D^2 \times \R)$, as illustrated in Figure~\ref{Fig:corked}.

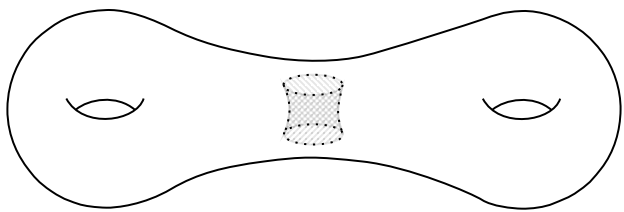
\begin{figure}[ht]
\usetikzlibrary{patterns}

\tikzset{
pattern size/.store in=\mcSize,
pattern size = 5pt,
pattern thickness/.store in=\mcThickness,
pattern thickness = 0.3pt,
pattern radius/.store in=\mcRadius,
pattern radius = 1pt}
\makeatletter
\pgfutil@ifundefined{pgf@pattern@name@_fctstrmzj}{
\pgfdeclarepatternformonly[\mcThickness,\mcSize]{_fctstrmzj}
{\pgfqpoint{0pt}{0pt}}
{\pgfpoint{\mcSize+\mcThickness}{\mcSize+\mcThickness}}
{\pgfpoint{\mcSize}{\mcSize}}
{
\pgfsetcolor{\tikz@pattern@color}
\pgfsetlinewidth{\mcThickness}
\pgfpathmoveto{\pgfqpoint{0pt}{0pt}}
\pgfpathlineto{\pgfpoint{\mcSize+\mcThickness}{\mcSize+\mcThickness}}
\pgfusepath{stroke}
}}
\makeatother


\tikzset{
pattern size/.store in=\mcSize,
pattern size = 5pt,
pattern thickness/.store in=\mcThickness,
pattern thickness = 0.3pt,
pattern radius/.store in=\mcRadius,
pattern radius = 1pt}
\makeatletter
\pgfutil@ifundefined{pgf@pattern@name@_xgt741hss}{
\pgfdeclarepatternformonly[\mcThickness,\mcSize]{_xgt741hss}
{\pgfqpoint{0pt}{-\mcThickness}}
{\pgfpoint{\mcSize}{\mcSize}}
{\pgfpoint{\mcSize}{\mcSize}}
{
\pgfsetcolor{\tikz@pattern@color}
\pgfsetlinewidth{\mcThickness}
\pgfpathmoveto{\pgfqpoint{0pt}{\mcSize}}
\pgfpathlineto{\pgfpoint{\mcSize+\mcThickness}{-\mcThickness}}
\pgfusepath{stroke}
}}
\makeatother
\tikzset{every picture/.style={line width=0.75pt}} 

\begin{tikzpicture}[x=0.75pt,y=0.75pt,yscale=-1,xscale=1]

\draw  [color={rgb, 255:red, 0; green, 0; blue, 0 }  ,draw opacity=1 ][dash pattern={on 0.84pt off 2.51pt}] (178.25,95.75) .. controls (178.25,92.93) and (184.92,90.65) .. (193.15,90.65) .. controls (201.38,90.65) and (208.05,92.93) .. (208.05,95.75) .. controls (208.05,98.57) and (201.38,100.85) .. (193.15,100.85) .. controls (184.92,100.85) and (178.25,98.57) .. (178.25,95.75) -- cycle ;
\draw  [draw opacity=0.1][pattern=_fctstrmzj,pattern size=2.25pt,pattern thickness=0.75pt,pattern radius=0pt, pattern color={rgb, 255:red, 0; green, 0; blue, 0}] (207.25,99.46) -- (206.67,102.04) -- (206.17,104.46) -- (205.92,106.71) -- (205.75,109.46) -- (205.92,111.96) -- (206.33,114.71) -- (207.33,118.04) -- (208.25,120.75) -- (207.17,122.71) -- (205.08,123.88) -- (203,124.63) -- (200,125.38) -- (197.17,125.63) -- (194.17,125.88) -- (191.58,125.88) -- (188.25,125.54) -- (184.67,124.96) -- (181.42,123.96) -- (179.08,122.13) -- (178.45,120.75) -- (179.83,116.63) -- (180.67,113.38) -- (181.33,108.88) -- (181,104.46) -- (180.08,100.63) -- (178.25,95.75) -- (179.92,98.13) -- (182.08,99.29) -- (184.67,100.04) -- (187.17,100.38) -- (190.75,100.88) -- (194.42,100.88) -- (197.92,100.63) -- (201.08,100.04) -- (204.5,98.96) -- (206.83,97.79) -- (208.05,95.75) -- cycle ;
\draw [color={rgb, 255:red, 0; green, 0; blue, 0 }  ,draw opacity=1 ] [dash pattern={on 0.84pt off 2.51pt}]  (208.05,95.75) .. controls (205.85,106.45) and (204.25,109.65) .. (208.25,120.75) ;
\draw  [color={rgb, 255:red, 0; green, 0; blue, 0 }  ,draw opacity=1 ][dash pattern={on 0.84pt off 2.51pt}] (178.45,120.75) .. controls (178.45,117.93) and (185.12,115.65) .. (193.35,115.65) .. controls (201.58,115.65) and (208.25,117.93) .. (208.25,120.75) .. controls (208.25,123.57) and (201.58,125.85) .. (193.35,125.85) .. controls (185.12,125.85) and (178.45,123.57) .. (178.45,120.75) -- cycle ;
\draw   (74.56,60.13) .. controls (79.48,58.75) and (82.58,58.33) .. (88.88,58) .. controls (95.17,57.67) and (105.89,59.5) .. (120.92,66.97) .. controls (135.94,74.44) and (149.04,77.75) .. (168.25,81.25) .. controls (187.46,84.75) and (206.25,84.02) .. (218.25,81.25) .. controls (230.25,78.48) and (272.25,64.9) .. (278.56,62.75) .. controls (284.87,60.6) and (288.28,58.92) .. (298.38,58.5) .. controls (308.47,58.08) and (318.73,62.49) .. (323.79,65.52) .. controls (328.84,68.55) and (333.02,71.52) .. (338.87,79.21) .. controls (344.71,86.9) and (348.4,98) .. (348.25,108.38) .. controls (348.1,118.75) and (345.17,129.52) .. (338.4,138.13) .. controls (331.63,146.75) and (331.02,146.29) .. (326.71,149.52) .. controls (322.4,152.75) and (318.25,153.98) .. (314.16,155.73) .. controls (310.08,157.48) and (303.25,158.37) .. (298.38,158.25) .. controls (293.5,158.13) and (283.48,157.06) .. (278.4,153.83) .. controls (273.33,150.6) and (241.94,137.07) .. (218.25,134.45) .. controls (194.56,131.83) and (188.16,131.8) .. (168.25,134.45) .. controls (148.34,137.1) and (135.14,140.31) .. (121.58,148.31) .. controls (108.03,156.31) and (93.94,157.86) .. (88.88,157.75) .. controls (83.81,157.64) and (77.63,156.9) .. (74.71,155.83) .. controls (71.79,154.75) and (68.1,154.13) .. (61.25,149.5) .. controls (54.4,144.87) and (52.3,141.98) .. (49.02,137.67) .. controls (45.74,133.37) and (38.87,122.9) .. (39,107.88) .. controls (39.13,92.85) and (45.6,82.69) .. (49.17,77.67) .. controls (52.75,72.65) and (56.84,69.67) .. (61.25,66.5) .. controls (65.66,63.33) and (69.63,61.52) .. (74.56,60.13) -- cycle ;
\draw    (68.72,102.65) .. controls (77.05,118.32) and (104.05,114.32) .. (107.72,102.65) ;
\draw    (73.58,108.22) .. controls (79.98,102.88) and (94.12,100.48) .. (103.32,108.35) ;

\draw    (278.78,102.72) .. controls (287.12,118.38) and (314.12,114.38) .. (317.78,102.72) ;
\draw    (283.65,108.28) .. controls (290.05,102.95) and (304.18,100.55) .. (313.38,108.42) ;

\draw [color={rgb, 255:red, 0; green, 0; blue, 0 }  ,draw opacity=1 ] [dash pattern={on 0.84pt off 2.51pt}]  (178.25,95.75) .. controls (181.85,103.85) and (182.65,110.25) .. (178.45,120.75) ;
\draw  [draw opacity=0.1][pattern=_xgt741hss,pattern size=2.25pt,pattern thickness=0.75pt,pattern radius=0pt, pattern color={rgb, 255:red, 0; green, 0; blue, 0}] (191,90.63) -- (194.25,90.63) -- (200,91.13) -- (203.92,92.21) -- (207.17,93.79) -- (208.05,95.75) -- (207.25,99.46) -- (206.67,102.04) -- (206.17,104.46) -- (205.92,106.71) -- (205.75,109.46) -- (205.92,111.54) -- (206.25,113.71) -- (206.83,116.04) -- (207.33,118.04) -- (208.25,120.75) -- (207.58,119.04) -- (205.5,117.79) -- (203,116.79) -- (200.08,116.04) -- (197.25,115.79) -- (193.35,115.65) -- (189.75,115.79) -- (186.08,116.21) -- (182.67,117.04) -- (180.17,118.29) -- (178.45,120.75) -- (179.08,119.21) -- (180.17,116.29) -- (180.67,113.38) -- (181.25,109.79) -- (181.25,107.13) -- (181,104.46) -- (180.08,100.63) -- (179.08,97.79) -- (178.25,95.75) -- (179.5,93.63) -- (182.67,92.21) -- (186.17,91.13) -- cycle ;

\end{tikzpicture}

\caption{The open handlebody $M_{F_\alpha}$ minus a tubular neighborhood of $L$.  The tubular neighborhood is the interior of the mapping torus of $F_\alpha'$ restricted to $Z$.}
\label{Fig:corked}
\end{figure}

For any $t \in (0,1)$ and $j \in \Z$, consider the disk
\[ \Omega_{j,t} = B_{j,1/2} \setminus int(D_{N,t} \cup D_{S,t}) = \pi_\alpha\left( \tfrac12 I_j \times [t-1,1-t]\right). \]
See Figure~\ref{Fig:shrinking disks}.
\begin{center}
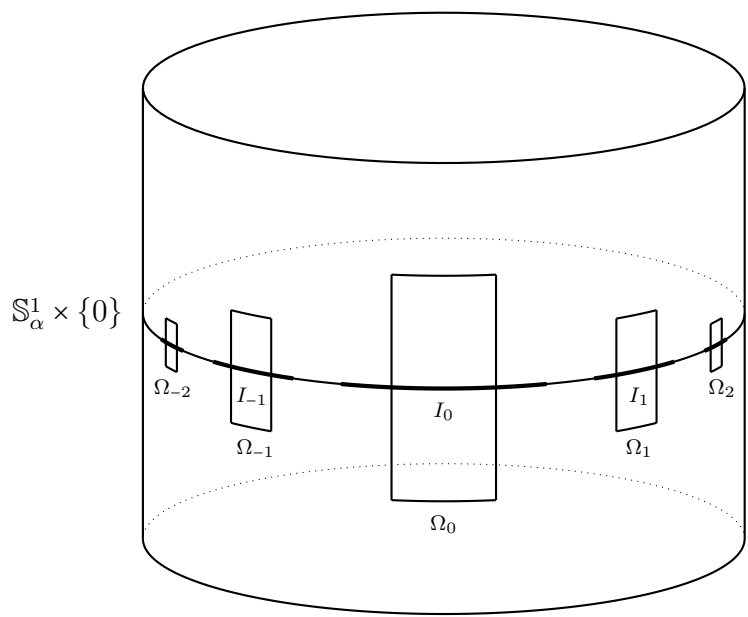
\begin{figure}[ht]
\begin{tikzpicture}[scale=1]
	\draw [thick,samples=100,domain=0:360] plot ({4*cos(\x)},{sin(\x)+3});
	\draw [thick,samples=100,domain=180:360] plot ({4*cos(\x)},{sin(\x)-3});
	\draw [dotted,samples=100,domain=0:180] plot ({4*cos(\x)},{sin(\x)-3});
	\draw [dotted,samples=100,domain=0:180] plot ({4*cos(\x)}, {sin(\x)});
	\draw [thick,samples=100,domain=180:360] plot ({4*cos(\x)}, {sin(\x)});
	\draw [thick] (-4,-3) -- (-4,3);
	\draw [thick] (4,-3) -- (4,3);
	\draw [ultra thick,domain=200:210] plot ({4*cos(\x)},{sin(\x)});
	\draw [ultra thick,domain=220:240] plot ({4*cos(\x)},{sin(\x)});
	\draw [ultra thick,domain=250:290] plot ({4*cos(\x)},{sin(\x)});
	\draw [ultra thick,domain=300:320] plot ({4*cos(\x)},{sin(\x)});
	\draw [ultra thick,domain=330:340] plot ({4*cos(\x)},{sin(\x)});
	\draw [thick,domain=202.5:207.5] plot ({4*cos(\x)},{sin(\x)+.32});
	\draw [thick,domain=202.5:207.5] plot ({4*cos(\x)},{sin(\x)-.32});
	\draw [thick] ({4*cos(202.5)},{sin(202.5)+.32}) -- ({4*cos(202.5)},{sin(202.5)-.32}) ;
	\draw [thick] ({4*cos(207.5)},{sin(207.5)+.32}) -- ({4*cos(207.5)},{sin(207.5)-.32}) ;
	\draw [thick,domain=225:235] plot ({4*cos(\x)},{sin(\x)+.75});
	\draw [thick,domain=225:235] plot ({4*cos(\x)},{sin(\x)-.75});
	\draw [thick] ({4*cos(225)},{sin(225)+.75}) -- ({4*cos(225)},{sin(225)-.75}) ;
	\draw [thick] ({4*cos(235)},{sin(235)+.75}) -- ({4*cos(235)},{sin(235)-.75}) ;
	\draw [thick,domain=260:280] plot ({4*cos(\x)},{sin(\x)+1.5});
	\draw [thick,domain=260:280] plot ({4*cos(\x)},{sin(\x)-1.5});
	\draw [thick] ({4*cos(260)},{sin(260)+1.5}) -- ({4*cos(260)},{sin(260)-1.5}) ;
	\draw [thick] ({4*cos(280)},{sin(280)+1.5}) -- ({4*cos(280)},{sin(280)-1.5}) ;
	\draw [thick,domain=305:315] plot ({4*cos(\x)},{sin(\x)+.75});
	\draw [thick,domain=305:315] plot ({4*cos(\x)},{sin(\x)-.75});
	\draw [thick] ({4*cos(305)},{sin(305)+.75}) -- ({4*cos(305)},{sin(305)-.75}) ;
	\draw [thick] ({4*cos(315)},{sin(315)+.75}) -- ({4*cos(315)},{sin(315)-.75}) ;
	\draw [thick,domain=332.5:337.5] plot ({4*cos(\x)},{sin(\x)+.32});
	\draw [thick,domain=332.5:337.5] plot ({4*cos(\x)},{sin(\x)-.32});
	\draw [thick] ({4*cos(332.5)},{sin(332.5)+.32}) -- ({4*cos(332.5)},{sin(332.5)-.32}) ;
	\draw [thick] ({4*cos(337.5)},{sin(337.5)+.32}) -- ({4*cos(337.5)},{sin(337.5)-.32}) ;
	\node at (-2.55,-1.1) {\tiny $I_{-1}$};
	\node at (0,-1.3) {\tiny $I_0$};
	\node at (2.6,-1.1) {\tiny $I_1$};
	\node at (-3.6,-1) {\tiny $\Omega_{-2}$};
	\node at (-2.5,-1.8) {\tiny $\Omega_{-1}$};
	\node at (0,-2.8) {\tiny $\Omega_{0}$};
	\node at (2.6,-1.8) {\tiny $\Omega_{1}$};
	\node at (3.7,-1) {\tiny $\Omega_2$};
	\node at (-5,0) {$\mathbb S^1_\alpha \times \{0\}$};
\end{tikzpicture}
\caption{The product $\mathbb S^1_\alpha \times [-1,1]$ that projects by $\pi_\alpha$ to $\mathbb S^2_\alpha$, together with a few rectangles that project to $\Omega_{-2} = \Omega_{-2,7/8},\Omega_{-1} = \Omega_{-1,3/4},\Omega_{0} = \Omega_{0,1/2},\Omega_{1} = \Omega_{1,3/4},\Omega_2 =\Omega_{2,7/8}$.}
\label{Fig:shrinking disks}
\end{figure}
\end{center}

There is an $F_\alpha$--invariant, disjoint union of disks given by
\[ \bigcup_{j \in \mathbb Z} \Omega_{j,t},\]
for any $t \in (0,1)$. We remark that $\bigcup_{j \in \mathbb Z} \Omega_{j,t}$  has accumulation points that do not belong to this set, and thus this union of disks is not properly embedded in $\Sigma_\alpha$. What is more, \emph{there is no $F_\alpha$--invariant, properly embedded surface in $\Sigma_\alpha$ that is homeomorphic to an infinite disjoint union of disks}. We leave the verification of this fact to the reader who is curious why we bother with the construction described over the next few pages.

Because of this phenomenon if we want to find a properly embedded and invariant surface which is homeomorphic to an infinite disjoint union of disks there is no option but to isotope $F_\alpha$ to a homeomorphism $F_\alpha'$ with the desired properties. The construction of this isotopy is rather technical and we perform it in what follows. After, we use it to give a proof of Theorem~\ref{thm:main2}.


 The desired isotopy is obtained as the composition of $F_\alpha$ with an infinite sequence of commuting isotopies $H^j_t \colon \Sigma_\alpha \to \Sigma_\alpha$ of the identity, each supported on one of the disks in the properly embedded subsurface
\[ \Omega \, \, = \, \, \bigcup_{j \in \mathbb Z} \Omega_j \, \, \subset \, \, \Sigma_\alpha, \]
where $\displaystyle{\Omega_j = \Omega_{j,(1-2^{-(|j|+1)})}}$.  More precisely, we will chose a disk $D \subset int(\Omega_0)$ and recursively construct isotopies $H^j_t \colon \Sigma_\alpha \to \Sigma_\alpha$, $t \in [0,1]$, such that
\begin{enumerate}
\item $H^j_0 = \mbox{id}_{\Sigma_\alpha}$ for all $j \in \mathbb Z$,
\item $H^j_t$ is supported on $\Omega_j$ for all $j \in \mathbb Z$; that is, $H^j_t = \mbox{id}$ outside $\Omega_j$, for all $t \in [0,1]$,
\item $H^j_t(y_j) = y_j$ for all $t \in [0,1]$ and $j \in \mathbb Z$,
\item $\mathcal H_t^k = H^{-k}_t \circ H^{-(k-1)}_t \circ \cdots \circ H^{k-1}_t \circ H^k_t$ converges locally uniformly to an isotopy $\mathcal H_t$ as $k \to \infty$ with $\mathcal H_0 = \mbox{id}_{\Sigma_\alpha}$.
\end{enumerate}
Furthermore, $F_\alpha' = \mathcal H_1 \circ F_\alpha$ will have the property that
\[ Z = \bigcup_{j \in \mathbb Z} (F_\alpha')^j(D) \subset \Omega \subset \Sigma_\alpha\]
is a properly embedded, $F_\alpha'$--invariant subsurface $Z \subset \Sigma_\alpha$ homeomorphic to a disjoint union of disks.\\

\noindent {\bf Construction:} To start, we will construct the isotopies $H^j_t$ for $j \geq 1$ using $F_\alpha$.  After that, we define $H^j_t$ for $j \leq 0$ by an essentially identical construction, but using $F_\alpha^{-1}$ instead.

We start by choosing any disk $D = D_0 \subset \Omega_0$ centered on $y_0$ so that $F_\alpha(D) \subset \Omega_1$.  Define $H^1_t$ to be the identity outside $\Omega_1$, so that
\begin{enumerate}
\item $H^1_0$ is the identity on $\Sigma_\alpha$,
\item $H^1_t$ fixes $y_1$ for all $t$, and
\item $F_\alpha(H^1_1(F_\alpha(D_0))) \subset int(\Omega_2)$.
\end{enumerate}
We can easily construct such an isotopy $H^1_t$ by sufficiently contracting everything in $int(\Omega_1)$ toward $y_1$ as $t$ varies from $0$ to $1$.  Set $D_1 = H^1_1(F_\alpha(D_0))$ and note that by condition (3) we have $F_\alpha(D_1) \subset int(\Omega_2)$.

Now inductively assume that for some $j \geq 1$, an isotopy $H^j_t$ supported on $\Omega_j$ has been defined, as has a disk $D_j \subset int(\Omega_j)$ centered $y_j$, such that
\begin{enumerate}
\item $H^j_0$ is the identity on $\Sigma_\alpha$,
\item $H^j_t$ fixes $y_j$ for all $t$, and
\item $F_\alpha(D_j) \subset int(\Omega_{j+1})$.
\end{enumerate}
The construction of $H_t^{j+1}$ is similar to the case $j=1$: contract everything in $int(\Omega_{j+1})$ toward $y_{j+1}$ as $t$ varies from $0$ to $1$ by a sufficient amount to ensure $F_\alpha(H^{j+1}_1(D_j)) \subset int(\Omega_{j+2})$.  We then define $D_{j+1} = H^{j+1}_1(F_\alpha(D_j))$.

The recursive procedure defines $H^j_t$ and disks $D_j$ for all $j \geq 1$, satisfying $H^{j+1}_1 F_\alpha(D_j) = D_{j+1}$ for all $j \geq 0$.  Furthermore, every $H^j_0$ is the identity, $H^j_t$ is the identity outside $\Omega_j$, and since the union of disks $\displaystyle{\bigcup_{j \geq 1} \Omega_j }$ is properly embedded, the limit of the composition of the isotopies
\[ \mathcal H_t^+ = \lim_{j\to \infty} H^1_t \circ H^2_t \circ \cdots \circ H^j_t,\]
exists.  Furthermore, for all $j \geq 1$ and $t \in [0,1]$ we have $\mathcal H^+_t|_{\Omega_j} = H^{j}_t|_{\Omega_j}$ and hence
\begin{eqnarray*}
(\mathcal H^+_1 F_\alpha)^j(D_0) & = & \left( \mathcal H^+_1 F_\alpha \right) \circ \left( \mathcal H^+_1 F_\alpha \right) \circ \cdots \circ \left( \mathcal H^+_1 F_\alpha \right) \circ \left( \mathcal H^+_1 F_\alpha \right)\circ \left( \mathcal H^+_1 F_\alpha \right)(D_0)\\
& = & \left( \mathcal H^+_1 F_\alpha \right) \circ \left( \mathcal H^+_1 F_\alpha \right) \circ \cdots \circ \left( \mathcal H^+_1 F_\alpha \right)\circ \left( \mathcal H^+_1 F_\alpha \right) \circ \left( H^1_1 F_\alpha \right)(D_0)\\
& = & \left( \mathcal H^+_1 F_\alpha \right) \circ \left( \mathcal H^+_1 F_\alpha \right) \circ \cdots \circ \left( \mathcal H^+_1 F_\alpha \right)\circ \left( \mathcal H^+_1 F_\alpha \right)(D_1)\\
& = & \left( \mathcal H^+_1 F_\alpha \right) \circ \left( \mathcal H^+_1 F_\alpha \right) \circ \cdots \circ \left( \mathcal H^+_1 F_\alpha \right)\circ \left(H^2_1 F_\alpha \right)(D_1)\\
& = & \left( \mathcal H^+_1 F_\alpha \right) \circ \left( \mathcal H^+_1 F_\alpha \right) \circ \cdots \circ \left( \mathcal H^+_1 F_\alpha \right)(D_2)\\
& = & \cdots \,\, = \,\, D_j.
\end{eqnarray*}

We carry out an identical construction of isotopies of the identity $G^{|j|}_t$ supported on $\Omega_j$, for $j \leq -1$, using $F_\alpha^{-1}$ and $D_0$ in place of $F_\alpha$ and $D_0$ (observe that since $F_\alpha(D_0) \subset \Omega_1$, we also have $F_\alpha^{-1}(D_0) \subset \Omega_{-1}$).  Then $D_j$ is defined by
\[ D_j = G^{|j|}_1 F_\alpha^{-1}(D_{j+1}) \]
for all $j \leq -1$.  Finally, for all $j \leq 0$, we set
\[ H^j_t  = F_\alpha (G^{|j|+1}_t)^{-1} F_\alpha^{-1}.\]
Observe that for all $j \leq 0$, $H^j_t$ is an isotopy of the identity supported on $F_\alpha(\Omega_{j-1}) \subset \Omega_j$.  Furthermore, for all $j \leq -1$ we have
\[ H^{j+1}_1 F_\alpha(D_j) = H^{j+1}_1 F_\alpha (G^{|j|}_1 \circ F_\alpha^{-1}(D_{j+1})) = (F_\alpha (G^{|j|}_1)^{-1}F_\alpha^{-1})F_\alpha(G^{|j|}_1 F_\alpha^{-1}(D_{j+1})) = D_{j+1}.\]
We can extract a limit
\[ \mathcal H_t^- = \lim_{j\to -\infty} H^0_t \circ H^{-1}_t \circ \cdots \circ H^j_t,\]
which is supported on $\displaystyle{\bigcup_{j \leq 0} \Omega_j}$ and has $\mathcal H_1^-\circ F_\alpha(D_j) = D_{j+1}$ for all $j \leq -1$.  The composition $\mathcal H_t = \mathcal H_t^+ \circ \mathcal H_t^-$ is the required isotopy, we set $F_\alpha' = \mathcal H_1\circ F_\alpha$, the isotopy of $F_\alpha$, and let
\[ Z = \bigcup_{j \in \mathbb Z} D_j \]
be the properly embedded, $F_\alpha'$--invariant subsurface which is an infinite, disjoint union of disks.  Observe that, $Z$ accumulates on the entire Cantor set of ends, $C_\alpha$. Moreover $C_\alpha\cup Z$ is the closure of $Z$ in the sphere $\mathbb{S}^2_\alpha$.\\

\textbf{Proof of Theorem~\ref{thm:main2}.} Observe that $V_2 \cong M_{F_\alpha} \cong M_{F_\alpha'}$.  Furthermore, because each $y_j$ was fixed throughout the isotopy $\mathcal H_t$, the homeomorphism from $M_{F_\alpha}$ to $M_{F_\alpha'}$ sends $L$ ``to itself", or more precisely to the suspension flowline through $y_0$,
\[ L' =  \bigcup_{s \in \mathbb R} \psi_s'(y_0) \subset M_{F_\alpha'}.\]
Furthemore, the suspension flow on $D_0$ produces a closed, properly embedded, tubular neighborhood of $L$,
\[ N(L') = \bigcup_{s \in \mathbb R} \psi_s'(D_0) \subset M_{F_\alpha'}.\]
Thus, $M_{F_\alpha'} \setminus int(N(L'))$ is homeomorphic to the manifold shown in Figure~\ref{Fig:corked}.
Furthermore, setting $Y_\alpha = \Sigma_\alpha \setminus int(Z)$ and $F^Y_\alpha = F_\alpha'|_{Y_\alpha} \colon Y_\alpha \to Y_\alpha$, we have
\[ M_{F^Y_\alpha} \cong M_{F_\alpha'} \setminus int(N(L)),\]
which is thus also homeomorphic to the manifold in Figure~\ref{Fig:corked}.

Next, observe that $\partial Y_\alpha$ is an infinite disjoint union of circles, indexed by $\mathbb Z$, and the homeomorphism $F^Y_\alpha$ restricted to $\partial Y_\alpha$ simply shifts one circle to the next.  Now let
\[ X = \bigsqcup_{j \in \mathbb Z} X_j \]
where each $X_j$ is homeomorphic to $S_{1,1}$, a genus $1$ surface with one boundary component, we can take $h \colon X \to X$ sending $X_j$ to $X_{j+1}$ ``by the identity" (any homeomorphism will do, as long as it shifts components by $1$), we can construct a homeomorphism $\varphi \colon \partial Y_\alpha \to \partial X$ conjugating $F^Y_\alpha|_{\partial Y_\alpha}$ to $h|_{\partial X}$.  Then
\[ S_\alpha = Y_\alpha \cup_\varphi X \]
is a surface with infinite genus, all ends accumulated by genus, and $\Ends(S) \cong C_\alpha$.  Therefore, $S_\alpha$ is homeomorphic to the blooming Cantor tree, again appealing to the classification of surfaces.  Moreover, the homeomorphisms $F^Y_\alpha$ and $h$ glue together to give a homeomorphism
\[ G_\alpha \colon S_\alpha \to S_\alpha \]
whose action on $\Ends(S)$ is conjugate to the restriction of $f_\alpha$ to $C_\alpha$.

Finally, the mapping torus of $G_\alpha$ is obtained by gluing $M_{F^Y_\alpha}$ to $M_h$,
\[ M_{G_\alpha} \cong M_{F^Y_\alpha} \cup M_h.\]
We know $M_{F^Y_\alpha}$ is homeomorphic to $V_2$ minus the interior of a closed tubular neighborhood of an unknotted, properly embedded arc (as in Figure~\ref{Fig:corked}), while $M_h \cong S_{1,1} \times \mathbb R$.  However, $S_{1,1}$ is a disk (2-dimensional 0-handle) with two 2-dimensional 1-handles attached.  Then we take the product with $[-1,1]$ to obtain a 3-ball (3-dimensional 0-handle) with two 3-dimensional 1-handles attached.  Up to homeomorphism, we can get $S_{1,1} \times \R$ from $S_{1,1} \times [-1,1]$ by removing $S_{1,1} \times \{-1,1\}$: what is left of the boundary is $\partial S_{1,1} \times \R$, which is an open annulus in the boundary.
See Figure~\ref{Fig:S11xR}.
Gluing $M_{F^Y_\alpha}$ to $M_h$ thus amounts to gluing two $1$--handles to $M_h$, and removing the boundary (i.e.~taking the interior); or equivalently, adding two $1$--handles to $V_2$, then removing the boundary.  This is precisely $V_4$, the interior of a genus $4$ handlebody, and thus
\[ M_{G_\alpha} \cong V_4.\]

\begin{center}
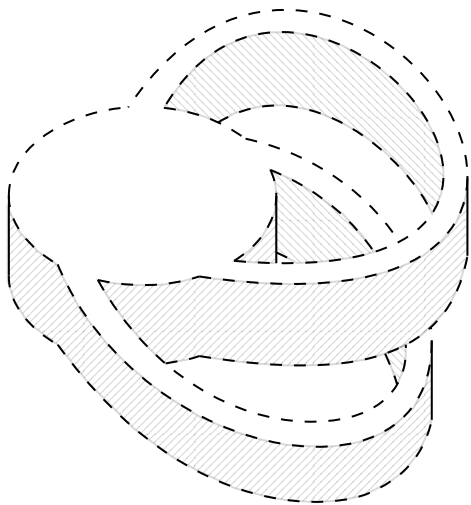
\begin{figure}[ht]
\vspace{-30mm}
\usetikzlibrary{patterns}
\tikzset{every picture/.style={line width=0.75pt}} 
\begin{tikzpicture}[x=0.75pt,y=0.75pt,yscale=-0.5,xscale=0.5]
\path (-280,588); 
\draw [color={rgb, 255:red, 0; green, 0; blue, 0 }  ,draw opacity=1 ]   (39.25,248) -- (39.25,328.38) ;
\draw [color={rgb, 255:red, 0; green, 0; blue, 0 }  ,draw opacity=1 ] [dash pattern={on 4.5pt off 4.5pt}]  (88.5,398.16) .. controls (223.5,582.75) and (370,562.25) .. (402,548.25) .. controls (434,534.25) and (456.5,513.75) .. (465.83,474.66) ;
\draw  [dash pattern={on 4.5pt off 4.5pt}]  (39.25,328.38) .. controls (38.57,356.57) and (60,382.1) .. (88.5,398.16) ;
\draw [color={rgb, 255:red, 0; green, 0; blue, 0 }  ,draw opacity=1 ]   (465.83,393.62) -- (465.83,473.99) ;
\draw  [dash pattern={on 4.5pt off 4.5pt}]  (88.5,317.12) .. controls (173.67,522.5) and (482.67,573.83) .. (465,380.08) ;
\draw  [dash pattern={on 4.5pt off 4.5pt}]  (39.25,247.33) .. controls (38.57,275.53) and (60,301.06) .. (88.5,317.12) ;
\draw [color={rgb, 255:red, 0; green, 0; blue, 0 }  ,draw opacity=1 ]   (309,248) -- (309,315.86) ;
\draw [color={rgb, 255:red, 0; green, 0; blue, 0 }  ,draw opacity=1 ] [dash pattern={on 4.5pt off 4.5pt}]  (266.44,313.56) .. controls (281.56,314.87) and (295.73,315.62) .. (309,315.86) .. controls (684.43,322.63) and (335.72,-81.5) .. (195.78,159.11) ;
\draw [color={rgb, 255:red, 0; green, 0; blue, 0 }  ,draw opacity=1 ] [dash pattern={on 4.5pt off 4.5pt}]  (231.33,329.33) .. controls (769.33,408) and (374.67,-161.33) .. (159.56,158.67) ;
\draw [color={rgb, 255:red, 0; green, 0; blue, 0 }  ,draw opacity=1 ]   (501.8,228) -- (501.8,308.38) ;
\draw [color={rgb, 255:red, 0; green, 0; blue, 0 }  ,draw opacity=1 ] [dash pattern={on 4.5pt off 4.5pt}]  (231.33,409.71) .. controls (334.8,425.6) and (479.6,437.2) .. (501.8,308.38) ;
\draw  [dash pattern={on 4.5pt off 4.5pt}]  (129,332.92) .. controls (184.14,430.94) and (268.97,473.4) .. (337.33,475.83) .. controls (405.69,478.26) and (442.67,446.83) .. (439.6,398.3) ;
\draw  [dash pattern={on 4.5pt off 4.5pt}]  (418.8,407.3) .. controls (423.8,419.3) and (426,421.7) .. (430.8,438.1) ;
\draw  [dash pattern={on 4.5pt off 4.5pt}]  (309,305.08) .. controls (316.5,308.58) and (320.67,310.25) .. (329.83,315.75) ;
\draw  [dash pattern={on 4.5pt off 4.5pt}]  (303,221.75) .. controls (357.83,243.92) and (388.83,276.92) .. (407.65,304.59) ;
\draw  [dash pattern={on 4.5pt off 4.5pt}]  (303,221.75) .. controls (306.33,228.42) and (309.17,237.92) .. (309,248) ;
\draw  [dash pattern={on 4.5pt off 4.5pt}]  (278,190.29) .. controls (335.5,202.75) and (395.33,236.58) .. (433,294.25) ;
\draw  [dash pattern={on 4.5pt off 4.5pt}]  (250,174) .. controls (340.67,118.22) and (441.78,211.11) .. (466.22,266) ;
\draw  [dash pattern={on 4.5pt off 4.5pt}]  (197.43,417.14) .. controls (214.29,413.71) and (213.14,415.71) .. (231.33,409.71) ;
\draw  [dash pattern={on 4.5pt off 4.5pt}]  (129,332.92) .. controls (179.78,343.78) and (213.11,334.89) .. (231.33,329.33) ;
\draw  [dash pattern={on 4.5pt off 4.5pt}]  (266.44,313.56) .. controls (305.4,289.3) and (309.2,258.9) .. (309,248) ;
\draw  [dash pattern={on 4.5pt off 4.5pt}]  (195.78,159.11) .. controls (242.17,164.75) and (263.17,180.25) .. (278,190.29) ;
\draw  [dash pattern={on 4.5pt off 4.5pt}]  (39.25,248) .. controls (39.11,192.44) and (109.56,161.33) .. (159.56,158.67) ;
\draw  [draw opacity=0.15][pattern={north east lines},pattern color={black}] (50,284.4) -- (68,303.6) -- (86.4,316) -- (112.8,365.6) -- (138.8,398) -- (158.4,417.6) -- (182.8,438.4) -- (200,451.2) -- (233.6,472) -- (267.2,486) -- (293.6,494) -- (318.4,499.2) -- (345.6,500.8) -- (361.2,501.2) -- (378.4,499.6) -- (400.4,494.4) -- (422.8,484) -- (447.2,462.4) -- (460,438) -- (466,408) -- (465.83,473.99) -- (462.8,484.8) -- (452.8,506.8) -- (437.6,526.4) -- (422,537.6) -- (396,550.8) -- (372.4,555.6) -- (348.4,557.2) -- (316.8,554.8) -- (298,551.2) -- (267.6,543.2) -- (225.6,523.2) -- (204.8,512) -- (166,483.5) -- (118.8,436) -- (88.5,398.16) -- (58.4,375.8) -- (44.4,353.2) -- (38,333.5) -- (38,259.5) -- cycle ;
\draw  [draw opacity=0.15][pattern={north east lines},pattern color={black}] (309,248) -- (309,315.86) -- (290.25,315.38) -- (266.44,313.56) -- (281.75,303.38) -- (293.75,290.63) -- (303.25,274.63) -- (307.25,265.63) -- cycle ;
\draw  [draw opacity=0.15][pattern={north west lines},pattern color={black}] (308.25,237.88) -- (306.5,230.13) -- (303,221.75) -- (327.75,232.63) -- (347.75,245.63) -- (359.74,254.32) -- (376,266.13) -- (395.25,287.13) -- (407.65,304.59) -- (396,308.13) -- (380,311.13) -- (353.25,314.63) -- (329.83,315.75) -- (321,310.88) -- (309,305.08) -- (309,248) -- cycle ;
\draw  [draw opacity=0.15][pattern={north west lines},pattern color={black}] (418.8,407.3) -- (439.6,398.3) -- (439.43,412) -- (437.14,424.29) -- (430.8,438.1) -- (424.29,418.86) -- cycle ;
\draw  [draw opacity=0.15][pattern={north east lines},pattern color={black}] (140.29,335.43) -- (158.29,337.71) -- (181.43,338.29) -- (195.43,337.43) -- (214,334.29) -- (231.33,329.33) -- (237.43,330.57) -- (261.43,333.14) -- (273.43,334.29) -- (294,336) -- (324.86,337.14) -- (352.86,336.29) -- (368.57,334.57) -- (385.43,332.29) -- (392.29,331.14) -- (404.86,328.57) -- (422.86,324) -- (442,316.29) -- (457.14,308.29) -- (476,293.14) -- (490.29,275.14) -- (498.22,256.22) -- (500.89,244.89) -- (501.8,228) -- (501.78,310.22) -- (492.57,340) -- (481.71,361.71) -- (465,380.08) -- (452.57,390.86) -- (439.6,398.3) -- (418.8,407.3) -- (399.43,413.14) -- (386.57,415.14) -- (369.71,417.71) -- (354.86,418.86) -- (318.86,419.14) -- (294.86,417.71) -- (272,415.43) -- (249.14,412.29) -- (237.43,410.57) -- (231.33,409.71) -- (217.14,413.71) -- (204,415.43) -- (197.43,417.14) -- (186.57,407.71) -- (175.43,396.29) -- (157.71,376) -- (140,351.43) -- (129,332.92) -- cycle ;
\draw  [draw opacity=0.15][pattern={north west lines},pattern color={black}] (460.29,170.57) -- (471.14,192.57) -- (477.43,217.71) -- (477.71,233.71) -- (475.2,246.4) -- (466.22,266) -- (459.43,252) -- (449.43,236.86) -- (435.71,220.57) -- (417.14,202.29) -- (396.57,186.57) -- (374.86,174) -- (346.29,161.71) -- (317.43,156.57) -- (298.57,156.57) -- (281.43,160.29) -- (267.14,164.29) -- (257.14,170.29) -- (250,174) -- (236,167.71) -- (216.57,162.29) -- (202.29,159.71) -- (195.78,159.11) -- (202.29,148.86) -- (212.57,135.14) -- (225.43,120.29) -- (238.29,108.86) -- (256,97.14) -- (278.86,87.43) -- (300.86,82.86) -- (322.57,82.86) -- (346.57,86) -- (375.71,96.29) -- (405.71,113.14) -- (428.57,132) -- (445.14,149.14) -- cycle ;
\end{tikzpicture}
\caption{The surface $S_{1,1} \times \mathbb R$. The shaded region is the open annular boundary.}
\label{Fig:S11xR}
\end{figure}
\end{center}


Since the action of $G_\alpha$ on $\Ends(Y_\alpha)$ is conjugate to $f_\alpha|_{C_\alpha}$, the discussion above and Corollary~\ref{cor:general conjugate equivalent} proves our second theorem from the introduction, precisely as in the proof of Theorem~\ref{thm:main1} at the beginning of \S\ref{section:proofs}.\\
\qed



\medskip

We can get many other examples, replacing $h \colon X \to X$ with other surfaces.  For example, we could let $X$  be one of the following cases:
\begin{enumerate}
	\item A countable union of copies of $S_{g,n,1}$, an orientable surface of genus $g$ with $n$ punctures and $1$ boundary component.
	\item A countable union of copies of $N_{g,n,1}$, a non-orientable surface of genus $g$\footnote{Recall that in the non-orientable case, the genus of a surface is the number of projective planes instead of tori.} with $n$ punctures and $1$ boundary component.
\end{enumerate}
Then, in both cases, let $h$ again shift these components transitively.

For the first case, by the same argument as above, $M_h$ becomes the interior of a genus $2g+n$ handlebody together with an open annular subsurface in the boundary.  The glued surface will be a blooming Cantor tree, minus a discrete set accumulating on every end of the blooming Cantor tree and the mapping torus will be homeomorphic to the interior of a genus $2g+n+2$ handlebody.

For the second case, using the same argument as above, $M_h$ becomes one of the following, depending on the genus:
\begin{itemize}
	\item a 3-ball with one non-orientable 1-handle, plus $n$ orientable handles, if $g = 1$,
	\item a 3-ball with one non-orientable 1-handle and one orientable 1-handle, plus $n$ orientable handles, if $g = 2$,
	\item a 3-ball with one non-orientable 1-handle and $g-1$ orientable 1-handles, plus $n$ orientable handles, if $g = 2r+1$ with $r>1$,
	\item and a 3-ball with two non-orientable 1-handles and $g-2$ orientable 1-handles, plus $n$ orientable handles, if $g = 2r+2$ with $r>1$.
\end{itemize}
Independently of the genus, in this case the glued surface will be a non-orientable blooming Cantor tree\footnote{In this context, a non-orientable blooming Cantor tree is the infinitely non-orientable surface of infinite genus with a Cantor set of ends, all of them accumulated by genus and non-orientable.}, minus a discrete set accumulating on every end of the blooming Cantor tree and the mapping torus will be homeomorphic to the interior of a handlebody with genus $g+n+2$ handles.

Finally, we remark that we can also combine the two constructions here and in the previous section to build many more examples with tame mapping torus.

\section{End spaces of mapping tori and questions} \label{sec:final thoughts}

We finish this manuscript by briefly discussing end spaces of mapping tori, including the proof of Theorem~\ref{thm:orbits to ends} and some open questions.

\subsection{End spaces}

The fact that every orbit of $F_{\alpha*}$ on the space of ends $\Ends(\Sigma_\alpha)$ is dense, and that the mapping torus $M_{F_\alpha}$ has one end is not an accident.  The next proposition mentioned in the introduction provides the link between action on the end space and space of ends of the mapping torus.

%
%


\begin{proposition}
	\label{P:Ends prop}
Suppose $f \colon S\to S$ is any homeomorphism of a connected surface $S$.  Then the inclusion $i \colon S \to M_f$, given by $i(x) = (x,0)$ admits a continuous extension to the end compactification.  Moreover, the induced map on end spaces $i_*\colon \Ends(S) \to \Ends(M_f)$ is a continuous, $f_*$--invariant surjection; that is, $i_* \circ f_* = i_*$.
\end{proposition}

\begin{proof}
Exhaust $S$ by compact $K_1 \subset K_2 \subset \ldots$ such that $K_i \subset \mbox{int}(K_{i+1})$ for all $i$ and so that $f(K_i) \subset \mbox{int}(K_{i+1})$.  Let $\pi \colon S \times [0,1] \to M_f$ be the quotient map, then $\mathcal K_i = \pi(K_i \times [0,1])$ is a compact exhaustion of $M_f$.  It follows that the inclusion of a fiber $i \colon S \to M_f$ is proper, so it has a continuous extension to the space of ends.
Note that $K_i\times[0,1]$ is a compact exhaustion of $S\times[0,1]$. Moreover, $\pi \colon S \times [0,1] \to M_f$ defines, for each $K_i$ in the exhaustion, a surjective map between the set of connected components of the complement of $K_i \times [0,1]$  and the set of connected components of the complement of $\mathcal K_i$. It follows that $i_*$ is surjective. Now if we have two (equivalence classes of) proper rays $\gamma\neq\delta$ in $\Ends(S)$ such that $\gamma=f(\delta)$, then the proper rays $\delta \times \{1\}$ and $\gamma \times \{0\}$ are equivalent (as end defining rays) in the mapping torus $M_f$. Hence $i_* \circ f_* = i_*$.
\end{proof}



\textbf{Proof of Theorem~\ref{thm:orbits to ends}}. Note that for any manifold $X$ the space $\Ends(X)$ is Hausdorff, see Theorem 1.5 in~\cite{Raymond60}. Hence the fibers of $\partial i$ are closed.  By Proposition~\ref{P:Ends prop} $\partial i$ is surjective and the fibers are $\partial f$--invariant.  The assumption that every $\partial f$--orbit is dense therefore implies that the fiber over any point is the entire endspace of $S$, hence $M_f$ must have exactly one end.\qed



\subsection{Questions}

We end this paper with some open questions.

\begin{question}
	\label{Q:1}
	Does there exist a fibration of $V_2$ over the circle where fibers are all homeomorphic to the blooming Cantor tree surface?
\end{question}

\begin{question}
\label{Q:2}
Does there exist an infinite type surface $S$ such that no mapping torus of it is tame? If the answer is yes, is there a classification of surfaces that admit homeomorphisms with tame mapping torus?
\end{question}

Unlike surfaces, $3$--manifolds can have finitely generated fundamental group, without being tame.  Indeed, the classical Whitehead manifold  is a contractible open $3$--manifold that is not tame; see, e.g.~\cite[Chapter 3]{Rolfsen}. Thus, we have the following two questions.

\begin{question}
\label{Q:3}
Are there necessary and sufficient conditions on a surface homeomorphism $f \colon S \to S$ that ensure that $M_f$ is tame?
\end{question}

It would be interesting to see if there are connections to criteria considered by Bestvina-Fanoni-Tao \cite{BesFanTao}.

\begin{question}
\label{Q:4}
Are there surface homeomorphisms $f \colon S \to S$ such that $\pi_1(M_f)$ is finitely generated, but $M_f$ is not tame?
\end{question}

  \bibliographystyle{alpha}
  \bibliography{mybib}

\end{document}